\documentclass[11pt]{article}
\usepackage[margin=1in]{geometry}
\usepackage{amsthm}

\usepackage{graphicx} 
\usepackage{amsmath,amssymb,amscd,mathtools}
\usepackage{dsfont}
\usepackage{xcolor}
\usepackage{todonotes}
\usepackage[normalem]{ulem} 
\usepackage{hyperref}
\usepackage[authoryear,round]{natbib}

\newcommand{\R}{\mathbb{R}}
\newcommand{\N}{\mathbb{N}}
\newcommand{\Z}{\mathbb{Z}}
\newcommand{\Aut}{\mathrm{Aut}}
\newcommand{\p}{\varphi}

\newcommand{\id}{\mathrm{id}}

\usepackage{cancel}

\newtheorem{theorem}{Theorem}[section]
\newtheorem{lemma}[theorem]{Lemma}
\newtheorem{corollary}[theorem]{Corollary}
\newtheorem{proposition}[theorem]{Proposition}
\newtheorem{definition}[theorem]{Definition}
\newtheorem{rem}[theorem]{Remark}
\newtheorem{es}[theorem]{Example}

\title{An Algebraic Representation Theorem for Linear GENEOs in Machine Learning}
\author{Francesco Conti\thanks{Inria Sophia Antipolis, DataShape team, France. Email: \texttt{francesco.conti@inria.fr}} \and 
Patrizio Frosini\thanks{University of Pisa, 56127 Pisa, Italy. Email: \texttt{patrizio.frosini@unipi.it}} \and 
Nicola Quercioli\thanks{Independent researcher. Email: \texttt{nicola.quercioli@gmail.com}}}
\date{}

\begin{document}

\maketitle

\abstract{
Geometric and Topological Deep Learning enhance machine learning by exploiting geometric and topological structures in data. In this context, Group Equivariant Non-Expansive Operators (GENEOs) provide a principled framework for encoding symmetries and designing efficient, interpretable neural architectures. Originally introduced in Topological Data Analysis and later adopted in Deep Learning, GENEOs unify geometric and topological approaches and include, as a special case, the operator computing persistence diagrams. Their theoretical foundations rely on group actions, equivariance, and compactness properties of operator spaces. GENEOs are defined between data spaces endowed with transformation groups, referred to as perception pairs. 
While previous representation results characterized linear GENEOs acting on identical perception pairs, many applications require operators between heterogeneous ones. In this work, we introduce and prove a new algebraic representation theorem for linear GENEOs acting between different perception pairs. Under mild assumptions on data domains and group actions, this result provides a complete characterization of such operators and implicitly yields a computational method to generate all linear GENEOs between two perception pairs. The proof combines tools from the theory of stochastic matrices, group actions, and invariant measures. In our paper, we also show that the space of linear GENEOs is compact and convex, making easier their use for learning and optimization. 
Our new theorem enlarges the applicability of GENEOs, and allows for their composition with a broader class of current architectures. We show this with an application in image reconstruction.
}

\medskip
\noindent\textbf{Keywords:} Group equivariant non-expansive operator, topological data analysis, group action, permutant measure

\medskip
\noindent\textbf{MSC 2020:} 15B51, 20C35, 47B38, 55N31, 62R40, 68T07, 68T09

\section{Introduction}
Geometric Deep Learning \cite{Br17,BrBr21} and Topological Deep Learning \cite{Ri19,Ho19,Ca20,We21,Zia2024} are two closely related and highly active research areas that have recently emerged to advance Machine Learning through topological and geometric techniques. Within this framework lies the notion of Group Equivariant Non-Expansive Operators (GENEOs).

Originally introduced about ten years ago in Topological Data Analysis to handle invariance of persistent homology groups \cite{FrJa16}, GENEOs later entered Deep Learning and Artificial Intelligence as geometric–topological tools for designing efficient neural networks \cite{bergomi2019towards,Mi23}. 
The key idea is twofold: first, replacing complex networks with smaller ones constructed from a limited number of equivariant components; and second, shifting attention from the structure of data spaces to that of agent/observer spaces, understood in terms of functional operators. 
By encoding symmetries through group actions, these operators can drastically reduce the number of parameters in AI applications while yielding interpretable Machine Learning models \cite{BocchiFMPPGGLBF24,CoBoGiGiPeFr25, AmCoQuZaMaGeFr26}.

A central feature of GENEOs is non-expansiveness. In the broader GENEO literature, non-expansiveness is crucial because, under suitable conditions, it guarantees the compactness of the space of operators, ensuring finite approximability and stability in learning procedures. While classical compactness results for GENEOs rely on compact function domains, the finite-dimensional setting considered here requires a different treatment. As we show in Section~\ref{Compactness}, linear GENEOs in our framework remain compact and, in fact, form a convex polytope, making them particularly well-suited for optimization. Non-expansiveness underpins other constructions as well: the distance between neural networks introduced in \cite{CoBoGiGiPeFr25} takes GENEOs as its building blocks, and their non-expansiveness is precisely what makes it well defined.

The connection between GENEOs and Geometric Deep Learning is most evident in their shared reliance on symmetries, group actions, and rich geometric structures, such as differentiable or Riemannian structures on manifolds. Notably, even the computation of the Laplacian and the multiplication by each circulant matrix $\mathrm{circ}(c_0,\ldots,c_{n-1})$ with $\sum_{i=0}^{n-1}|c_i|\le 1$ (to give just two simple examples) can be expressed as GENEOs. Equally compelling is the link with Topological Deep Learning, which builds on topological techniques (for instance, compactness theorems on spaces of operators) and on the remarkable fact that computing persistence diagrams in Topological Data Analysis can be viewed as a special case of a GENEO (see Section~\ref{msG}). Taken together, these observations underscore the role of GENEOs as a unifying framework bridging Geometric and Topological Deep Learning, thereby providing a cross-cutting tool for incorporating topological and geometric contributions into Machine Learning research.

Interest in GENEOs has steadily grown, leading to both a solid mathematical theory and practical applications across scientific domains (cf., e.g., the recent works \cite{Felipe2025,BoFeFr25,BoFrMi25,LaMiBoFrSo25,Ah26}). Once their usefulness in Machine Learning was established, a central challenge arose: designing equivariant operators with prescribed symmetries. A key question in this context is as follows: How can all linear GENEOs be computed, so as to make them available for applications in Machine Learning and Data Analysis? In \cite{BoBoBrFrQu23}, a theorem was established that proves a one-to-one correspondence between certain measures, known as permutant measures, and linear GENEOs mapping a dataset, represented by real-valued functions on a discrete space, onto itself. In many practical applications, however, the input and output data types differ, making the original theorem inapplicable. In this work, we extend the result to this broader setting by employing the generalized permutant measures introduced in \cite{FaFeFr2023}. Our main contribution is a new and more general representation theorem, valid for GENEOs acting between different perception pairs.

\subsection*{Relation to Other Constructions of Equivariant Operators}
Our model offers several advantages over classical methods for generating operators from $\R^n$ to $\R^m$ that are equivariant with respect to given homomorphisms $T\colon G\to K$, where $G$ is a group acting on $X=\{1,\ldots,n\}$ and $K$ is a group acting on $Y=\{1,\ldots,m\}$. Classical results show that every equivariant operator $F$ of this kind can be written as
\[
F(\phi)(y)
=
\int_X k(x,y)\phi(x)\,\lambda(dx),
\]
where $k(gx,T(g)y)=k(x,y)$ for every $g\in G$, and $\lambda$ is a suitable ``canonical'' measure on $X$, unique up to scaling. However, this construction merely reduces the problem to that of producing kernels with the required property, without suggesting simple procedures for obtaining them. On the other hand, the classical construction based on averaging suitable operators indexed by the elements of $G$ also entails the difficulty of computing integrals over potentially very large groups.

Our paper instead proposes producing equivariant operators by integrating neither over the domain $X$ nor over the group $G$, but over the space of functions from $Y$ to $X$, with respect to a signed measure $\mu$ on $X^Y$ such that $\mu(H) = \mu\bigl(gHT(g^{-1})\bigr)$ for every $g\in G$ and every measurable set $H\subseteq X^Y$. In many cases, it is easy to find a measure $\mu$ that satisfies the required
property and has relatively small support, thereby reducing the number of terms
that must be computed in the sum representing our discrete integral.
We refer to Example~\ref{newexample3} for an instance of this.

\subsection*{Main Contributions}
Our main contribution is a representation theorem that enables the computation of all linear GENEOs between two given perception pairs, assuming that the data functions have finite domains and that the corresponding group actions are transitive. Beyond its theoretical value, the result carries practical significance, which we demonstrate by applying it to enhance the performance of autoencoders. We list our contributions in detail below.
\begin{itemize}
    \item \textbf{Characterization of Linear GEOs:} We establish a general representation theorem that characterizes all linear Group Equivariant Operators (GEOs) between different perception pairs via
generalized $T$-permutant measures, providing a computational framework that, in several applications, can be more efficient than alternative approaches.
    \item \textbf{Characterization of Linear GENEOs:} We specialize this characterization to Group Equivariant Non-Expansive Operators (GENEOs) by considering the case in which $\mu$ is a signed measure
whose total variation norm is bounded by $1$.
    \item \textbf{Compactness Proof:} We prove that the space of linear GENEOs is compact, relying explicitly on their non-expansive nature in the finite-dimensional setting.
    \item \textbf{Convex Polytope Structure:} We demonstrate that the space of linear GENEOs forms a convex polytope and explicitly identify its extremal points, characterizing its fundamental building blocks.
    \item \textbf{Practical Application:} We illustrate the practical effectiveness of our framework by integrating  linear GENEOs into an autoencoder architecture for data reconstruction under unseen noise distributions.
\end{itemize}

\subsection*{Outline of the Paper}
The paper is organized as follows. Sections~\ref{Preliminaries} and \ref{Building} review the necessary definitions and preliminary results, setting the stage for Section~\ref{Representation}, where we introduce a representation theorem for linear group equivariant operators (Theorem~\ref{maint}). 
Building on this, Section~\ref{Linear} 
presents and proves our main result, which characterizes linear GENEOs 
(Theorem~\ref{teo:linear}). Section~\ref{Compactness} establishes the compactness and convexity of the space of all linear GENEOs under our assumptions. Section~\ref{Application} then provides an application to autoencoders, demonstrating how the operators introduced by our new representation theorem can be effectively leveraged in machine learning. Finally, Section~\ref{Conclusions} concludes the paper and outlines directions for future research.

\section{Preliminaries}
\label{Preliminaries}

In Sections~\ref{Preliminaries} and \ref{Building}, we recall the main concepts underlying this paper, including GEO and GENEO, as presented in \cite{FaFeFr2023}.

\subsection{Mathematical settings and GENEOs}\label{msG}
Assuming that $Z=\left\{ z_1, \dots, z_\ell \right\} = \left\{ z_k \right\}_{k=1}^\ell$ is a finite set with cardinality $\ell$, we can consider the space $\R^Z:= \{\varphi \colon Z \to \R \}$ of all real-valued functions on $Z$, equipped with the usual uniform norm $\left\| \cdot \right\|_\infty$. We recall that, if $Z$ is finite, $\R^Z \cong \R^\ell$, and $\R^Z$ has the canonical basis $\mathcal{B}_{Z}=\left( \mathds{1}_{z_1}, \dots, \mathds{1}_{z_\ell} \right)$, where for any $k = 1, \dots, \ell$ the function $\mathds{1}_{z_k} \colon X \to \R$ takes value $1$ on $z_k$, and $0$ otherwise. Moreover, $\Aut(Z)$ denotes the group of bijections from $Z$ to itself with respect to the usual composition operation. Although this paper is devoted to the representation of GENEOs in the discrete case, we retain, for the sake of simplicity and uniformity, the notation and terminology used in the classical continuous setting. $\Aut(Z)$  acts on $\R^Z$ by composition on the right. Considering a group $G \subseteq \Aut(Z)$ we say that the couple $\left( \R^Z, G \right)$ is a \textit{perception pair}, as usual in GENEO literature. To avoid overly heavy notation, in what follows we will juxtapose functions both to denote function composition and to denote their product. The context will make the meaning of the symbols clear.

\begin{definition}
    Given two finite sets $X,Y$, we consider two perception pairs $\left(\R^X, G\right)$, $\left(\R^Y, K\right)$, where $G \subseteq \Aut(X)$ and $K\subseteq\Aut(Y)$. A pair $(F, T)$ consisting of a map $F \colon \R^X \to \R^Y$ and a homomorphism $T \colon G \to K$, such that $F$ is $T$-equivariant (i.e., $F(\p g) = F(\p)T (g)$ for every $\p \in \R^X, g \in G$), is said to be a \textbf{Group Equivariant Operator} (GEO) from $\left(\R^X, G\right)$ to $\left(\R^Y, K\right)$.
\end{definition}

In our context, $F \colon \R^X \to \R^Y$ is non-expansive if for any $\varphi_1, \varphi_2 \in \R^X$:
\[\lVert F(\varphi_1) - F(\varphi_2)\rVert_\infty \le \lVert \varphi_1 - \varphi_2\rVert_\infty.\]
\begin{definition}
A GEO $\left(F, T\right)$ such that $F$ is non-expansive is called a \textbf{Group Equivariant Non-Expansive Operator} (GENEO).
\end{definition}
We say that a GENEO $(F, T)$ is \textit{linear} if the map $F \colon \R^X \to \R^Y$ is a linear map.

\begin{rem}
\label{rem:PH}
The concept of a GENEO (used in the general setting where finiteness of the signal domains is not required; see, e.g., \cite{bergomi2019towards}) provides a flexible framework for incorporating different operators into a unified setting. A key example in Topological Data Analysis is the computation of persistence diagrams. Let $X$ be a compact metric space, $\Phi$ the set of all continuous functions from $X$ to $\R$, and $\mathrm{Homeo}(X)$ the group of all homeomorphisms of $X$. For a fixed homological degree $k$, 
we consider the support of the persistence diagram of $\varphi$ in degree $k$, which is a compact subset of the extended half-plane $\bar\Delta^*$, with respect to a suitable metric (see~\cite{frosini2025matching} for details). 
Assuming that the persistence diagram under consideration has no multiple points, we observe that it can be represented both by its support and by the distance function from its support, without any loss of information. In this way, the set
$\mathrm{DGM}_k := \{ \mathrm{dgm}_k(\varphi) \mid \varphi \in \Phi \}$
can be interpreted both as a collection of compact sets and as a space of functions.
Equipping it with the Hausdorff distance $d_H$, one can verify that 
$(\mathrm{PD}_k,\mathbf{0})$ is a GENEO, where $\mathrm{PD}_k \colon \Phi \to \mathrm{DGM}_k$, $\mathrm{PD}_k \colon \varphi \mapsto \mathrm{dgm}_k(\varphi)$ and $\mathbf{0} \colon \mathrm{Homeo}(X) \to \{\id\}$ is the trivial homomorphism. The equivariance of $(\mathrm{PD}_k,\mathbf{0})$ corresponds to the invariance of persistence diagrams under reparameterizations of $X$, while non-expansiveness corresponds to stability under perturbations of $\varphi$ in the max norm.
\end{rem}

\subsection{Decomposition of stochastic matrices}
This section is devoted to recalling some results on stochastic matrices, mainly adapted from \cite{doi:10.1080/03081087.2020.1733461}. This background will be explicitly used to establish our main representation results in Section~\ref{Representation}. Let $\mathcal{M}_{m \times n}$ be the set of all $m \times n$ real-valued matrices. Moreover, with the symbol $\underline{\ell}$ we denote the set $\{1, \dots, \ell\}$ for every $\ell \in \N \setminus \{0\}$.

\begin{definition}
A matrix $A = (a_{ij}) \in \mathcal{M}_{m \times n}$ is \textbf{(right) stochastic} if $a_{ij} \ge 0$ for all $(i,j) \in \underline{m} \times \underline{n}$ and $\sum_{j=1}^n a_{ij}=1$ for all $i \in \underline{m}$.
\end{definition}

A $\{0,1\}$-matrix is a matrix $A=(a_{ij}) \in \mathcal{M}_{m \times n}$, such that $a_{ij} \in \{0,1\}$ for all $(i,j) \in \underline{m} \times \underline{n}$. We shall refer to the $\{0,1\}$-matrices in $\mathcal{M}_{m \times n}$ with exactly one $1$ in each row as \textbf{rectangular (row) permutation matrices} and $\mathcal{RP}_{m \times n}$ is the set of all rectangular permutation matrices of dimension $m \times n$. 
From now on, we will drop the right and row dependence from the definition of stochastic and rectangular permutation matrices, since it is understood. It is a well-known fact that the set of stochastic matrices is a convex set. Before proceeding we recall Theorem $1$ of \cite{doi:10.1080/03081087.2020.1733461}:

\begin{theorem}
\label{thmrectangular}
Every $m \times n$ stochastic matrix can be expressed as a convex combination of $m \times n$ rectangular permutation matrices.
\end{theorem}

This result directly motivates our focus on rectangular permutation matrices. Indeed, the convex combination established by this theorem anticipates the structure of the linear operators that we will construct. As we will see, such operators are essentially built via a measure $\mu$ acting as a weighting over the set of rectangular permutation matrices $\mathcal{RP}$, taking the form $F = \sum_{P \in \mathcal{RP}} \mu(P) P$.

\begin{rem}\label{notunique}
    In general, the convex combination stated in Theorem~\ref{thmrectangular} is not unique. As an example, let us consider the following stochastic matrix:
    \[ B = 
    \begin{pmatrix}
        1/2 & 0 & 1/2 \\
        1/3 & 1/3 & 1/3
    \end{pmatrix}.
    \]
    To keep the notation simple, we refer to the $2\times 3$ rectangular permutation matrix with entries $1$ in positions $(1,j_1)$ and $(2,j_2)$ as $R_{j_1,j_2}$. Using this notation, $B$ can be expressed as the following convex combinations:
    \[
        B = \frac{1}{12}R_{1,1} + \frac{5}{24}R_{1,2} + \frac{5}{24}R_{1,3} + \frac{1}{4}R_{3, 1} + \frac{1}{8}R_{3, 2} + \frac{1}{8}R_{3,3}
    \]
    and
    \[
        B = \frac{5}{24}R_{1,1} + \frac{1}{12}R_{1,2} + \frac{5}{24}R_{1,3} + \frac{1}{8}R_{3, 1} + \frac{1}{4}R_{3, 2} + \frac{1}{8}R_{3,3}.
    \]
\end{rem}

\section{Building linear GEOs via \texorpdfstring{$T$}{T}-permutant measures}
\label{Building}

\sectionmark{Building linear GEOs}
This section builds a linear GEO from each generalized $T$-permutant measure. These measures live on the set of functions from $Y$ to $X$, and are those that the equivariance group leaves unchanged; from each of them the operator is obtained by summing the reparameterizations of the input, weighted by the measure.
We consider the finite sets
$X = \left\{x_1, \dots, x_n\right\}
= \left\{x_j\right\}_{j=1}^n$
and
$Y = \left\{y_1, \dots, y_m\right\}
= \left\{y_i\right\}_{i=1}^m$,
together with the corresponding function spaces $\R^X$ and $\R^Y$.
While in this finite setting $\R^X$ and $\R^Y$ are naturally isomorphic to $\R^n$ and $\R^m$, we preserve the functional notation to maintain consistency with prior GENEO literature. We fix a group homomorphism $T \colon G \to K$, where $G \subseteq \Aut(X)$ and $K \subseteq \Aut(Y)$. We denote with $X^Y$ the set of all functions from $Y$ to $X$. For each $h \in X^Y$, we consider the following action of $G$ on $X^Y$:
\begin{equation*}
\label{eq:action}
    \alpha_T \colon G \times X^Y \to X^Y, \quad \left( g, h \right) \mapsto g h T(g^{-1}).
\end{equation*}

In particular, we are interested in the orbits $\mathcal{O}(h)$ of $h$ under the action $\alpha_T$, for $h \in X^Y$. Moreover, with a slight abuse of notation, we denote with $\mu(h)$ the signed measure of the singleton $\left\{ h \right\}$ for each $h \in X^Y$.

\begin{definition}
    A finite signed measure $\mu$ on the power set of $X^Y$ is called a \textbf{(generalized) $T$-permutant measure} if
    $\mu$ is invariant under the action $\alpha_T$ of $G$, i.e. $\mu(H) = \mu(gHT(g^{-1}))$ for every $g \in G$. Equivalently, we can say that a finite signed measure $\mu$ on the power set of $X^Y$ is a generalized $T$-permutant measure if
    $\mu(\left\{h\right\}) = \mu(\left\{ghT(g^{-1})\right\})$ for every $g \in G$.
\end{definition}

\begin{es}\label{newexample1}
Set $X=\underline{m}\times\underline{n}$ and $Y=\underline{m}$.
We fix $G$ to be the group of all permutations $g$ of $X$ that can be written as
$g=(k,k')$,
where $k$ is a permutation of $Y$ and $k'$ is a permutation of $Z=\underline{n}$, so that 
$g(i,j)=(k(i),k'(j))$.
Then we choose $K$ to be the group of all permutations of $Y$.
The set $X$ can be identified with the set of all cells of an $m\times n$ matrix, while $Y$ represents the set of its rows. Therefore, each $h\in X^{Y}$ can be seen as a map that associates to each row of an $m\times n$ matrix $M$ a cell of $M$.
Consider the homomorphism $T\colon G\to K$
defined by $T(k,k')=k$.
Define a finite signed measure $\mu$ on $X^{Y}$ as follows: for $h\in X^{Y}$, set $\mu(h)=|\operatorname{Im}(h)|$. Then $\mu$ is a generalized $T$-permutant measure.
\end{es}

\begin{es}\label{newexample2}
Using the same notation as in Example~\ref{newexample1}, define a finite signed measure
$\mu$ on $X^{Y}$ as follows: for $h \in X^{Y}$, set
$\mu(h) = 1$ if $h$ maps each row of $M$ into a cell of that row,
and $\mu(h) = 0$ otherwise.
Then $\mu$ is a generalized $T$-permutant
measure.
\end{es}

Thanks to the concept of generalized $T$-permutant measure, we are able to define a linear GEO associated with it. In fact, we can state the following.

\begin{proposition}
    \label{prop:GEOperm}
    Let $\mu$ be a generalized $T$-permutant measure. We define the operator $F_\mu \colon \R^X \to \R^Y$ associated with the measure $\mu$ by setting
    \[
    F_\mu \left( \p \right) := \sum_{h \in X^Y} \p h \ \mu(h)
    \]
    for every $\p \in \R^X$. Then $(F_\mu,T)$ is a linear GEO, where $F_\mu \colon \R^X \to \R^Y$ and $T \colon G \to K$.
\end{proposition}
\begin{proof}
    First, we show that $F_\mu$ is $T$-equivariant. It holds that:
    \begin{align*}
        F_\mu \left( \p g \right) & = \sum_{h \in X^Y} \p g h \ \mu(h) \\
        & = \sum_{h \in X^Y} \p g h T (g^{-1}) T (g) \ \mu\left(g h T(g^{-1})\right) \\
        & = \sum_{f \in X^Y} \p f T(g)\ \mu(f) \\
        & = F_\mu \left( \p \right) T(g),
    \end{align*}
    since $\mu(h) = \mu \left( ghT(g^{-1}) \right)$ and the map $h \mapsto f := g h T(g^{-1})$ is a bijection from $X^Y$ to $X^Y$. The linearity of $F_\mu$ directly comes from the definition.
\end{proof}

\begin{es}\label{newexample3}
Returning to the setting of Example~\ref{newexample2}, we show the advantages of constructing equivariant operators by means of Proposition~\ref{prop:GEOperm}. We recall that the measure $\mu$ gives mass $1$ to each of the $n^m$ maps sending every row of $M$ into a cell of that row, and mass $0$ to all the others. Then, we can define the linear GEO as
\[
F_\mu(\p) = \sum_{h \in X^Y} \p h \ \mu(h) = \sum_{h \in H} \p h \ \mu(h),
\]
where $H$ denotes the support of $\mu$. This is the main benefit of our approach. Producing the operator only requires us to exhibit a measure and check its invariance, and the sum ranges over the support of the measure, which here has $n^m$ elements. The classical group-averaging construction would instead involve all
$m!\,n!$ elements of $G$. It can be verified that $n^m \le n!\,m!$ holds for all positive integers $n,m$. Already for $n,m\sim 10$, the inequality becomes strongly asymmetric: $n^m \ll n!\,m!$.
\end{es}

\section{A representation theorem for linear GEOs via \texorpdfstring{$T$}{T}-permutant measures} 
\label{Representation}

Proposition~\ref{prop:GEOperm} provides a first link between generalized $T$-permutant measures and GEOs, and its definition and proof was straightforward. In contrast, proving the converse of the representation theorem, i.e. associating a generalized $T$-permutant measure with a GEO, is more difficult. In any case, we are able as of now to state the result that we will prove.

\begin{theorem}
    \label{maint} 
    Assume that $G \subseteq \mathrm{Aut}(X), K \subseteq \mathrm{Aut}(Y)$, and $T \colon G \to K$ is a group homomorphism. Moreover, suppose that $T(G)$ transitively acts on the finite set $Y$, and $F$ is a map from $\mathbb{R}^X$ to $\mathbb{R}^Y$. The map $(F,T)$ is a linear GEO from $(\mathbb{R}^X, G)$ to $(\mathbb{R}^Y, K)$ if and only if a $T$-permutant measure $\mu$ exists such that $F(\varphi) = \sum_{h \in X^Y} \varphi h\  \mu(h)$ for every $\varphi \in \mathbb{R}^X$.
\end{theorem}

\begin{rem}
In the case where the action of $T(G)$ on $Y$ is not transitive, 
this action decomposes $Y$ into several distinct orbits $Y_r$.
Even in this case, however, we can take advantage of Theorem~\ref{maint}, applying it to study each restriction $F(\p)_{\vert Y_r}:Y_r\to\R$ as the orbit $Y_r$ varies. This is possible since the action of $T(G)$ on $Y_r$ is clearly transitive.
\end{rem}

Throughout this paper, we use $\varphi$ to denote a function
$\varphi : X \to \mathbb{R}$, and $[\varphi]$ to denote its associated column
vector representation with respect to the standard basis of $\mathbb{R}^X$.
Similarly, we use $F$ to denote a linear map
$F : \mathbb{R}^X \to \mathbb{R}^Y$, and $[F]$ to denote its associated matrix
with respect to the standard bases $\mathcal{B}_X=([{\mathds{1}}_{x_1}],\ldots,[{\mathds{1}}_{x_n}])$ of $\R^X$ and $\mathcal{B}_Y=([{\mathds{1}}_{y_1}],\ldots [{\mathds{1}}_{y_m}])$ of $\R^Y$. The remainder of this section is devoted to the proof of Theorem~\ref{maint}.
We split the proof into several substeps, both to improve readability and to present each step as a separate result. We stress the fact that one implication has already been proved by Proposition~\ref{prop:GEOperm}. Let us assume that $(F, T)$ is a linear GEO, where $F \colon \R^X \to \R^Y$ and $T \colon G \to K$. Moreover, let $B=(b_{ij})$ be the matrix associated with $F$ with respect to the standard bases of $\R^X$ and $\R^Y$. The meanings of the symbols \(B\) and \([F]\) are therefore the same, but we will use the latter when we wish to emphasize that it is the matrix associated with the chosen operator \(F\). Given a natural number \(\ell \in \mathbb{N} \setminus \{0\}\) and a set \(Z\) with \(\lvert Z \rvert = \ell\), for every permutation \(p \colon Z \to Z\) we denote by \(\sigma_p \colon \underline{\ell} \to \underline{\ell}\) the function defined by setting \(\sigma_p(j) = i\) if and only if \(p(z_j) = z_i\). We observe that \(\sigma_{p^{-1}} = \sigma_p^{-1}\) and that
$\mathds{1}_{x_j} g = \mathds{1}_{g^{-1}(x_j)}$.
A similar identity holds for \(\mathds{1}_{y_i} T(g)\).

\begin{lemma}\label{rijhm}
For any $g \in G$, we have that $b_{ij}=b_{\sigma_{T(g)}(i)\sigma_g(j)}$ for every $(i,j) \in \underline{m} \times \underline{n}$.
\end{lemma}

\begin{proof}
Let us choose a function ${\mathds{1}}_{x_j}$ and a permutation $g \in G$.
By equivariance we have that
\[{F}({\mathds{1}}_{x_j}g )= {F}({\mathds{1}}_{x_j})T(g).\]
The left-hand side of the equation can be rewritten as:
\[
{F}({\mathds{1}}_{x_j}g )= {F}({\mathds{1}}_{g^{-1}(x_j)} )= \sum_{i=1}^{m}b_{i\sigma_g^{-1}(j)}{\mathds{1}}_{y_i}.
\]
On the right-hand side, we get
\begin{align*}
{F}({\mathds{1}}_{x_j})T(g)&=\left(\sum_{i=1}^{m}b_{ij}{\mathds{1}}_{y_i}\right) T(g)
\\&= \sum_{i=1}^{m}b_{ij}({\mathds{1}}_{y_i}T(g))
\\&= \sum_{i=1}^{m}b_{ij}({\mathds{1}}_{T(g^{-1})(y_i)})
\\&=\sum_{s=1}^{m}b_{\sigma_{T(g)}(s)j}{\mathds{1}}_{y_s},
\end{align*}
by setting $y_s=T(g^{-1})(y_i)$, and hence $i=\sigma_{T(g)}(s)$.
Therefore, we obtain the following equation: 
\[
\sum_{i=1}^{m}b_{i\sigma_g^{-1}(j)}{\mathds{1}}_{y_i}=\sum_{s=1}^{m}b_{\sigma_{T(g)}(s)j}{\mathds{1}}_{y_s}.
\]
This immediately implies that $b_{i\sigma_g^{-1}(j)}=b_{\sigma_{T(g)}(i)j}$, for any $i \in \underline{m}$.
Since this equality holds for
any $j \in \underline{n}$ and any $g \in G$, we have that $b_{ij}=b_{\sigma_{T(g)}(i)\sigma_g(j)}$ for every $(i,j) \in \underline{m} \times \underline{n}$ and every $g\in G$.
\end{proof}

For the rest of the section, we will assume that $T(G)$ transitively acts on $Y$. These assumptions allow us to prove the following result.

\begin{lemma}\label{lempermutationrow}
If the action of $T(G) \subseteq K$ on $Y$ is transitive, an $n$-tuple of real numbers $\beta=(\beta_1,\ldots,\beta_n)$ exists such that each row of ${B}$ can be obtained by permuting $\beta$.
\end{lemma}
\begin{proof}
Since $T(G)$ acts transitively, for every $i \in \underline{m}$ there exists
$k_{i \mapsto 1}\in T(G)$ such that $k_{i\mapsto1}(y_i)=y_1$.
Considering the $\overline \imath$-th row of $B$ and $k_{\overline \imath \mapsto 1} \in T(G)$, there exists $\overline g \in G$ such that $T(\overline g)=k_{\overline \imath \mapsto 1}$. Hence, by Lemma~\ref{rijhm}, we have that $b_{\overline \imath j}
=b_{\sigma_{T (\overline g)}(\overline \imath) \sigma_{\overline g}(j)}
=b_{\sigma_{k_{\overline\imath \mapsto 1}}(\overline \imath) \sigma_{\overline g}(j)}
=b_{1\sigma_{\overline g}(j)}$, for any $j \in \underline{n}$.
Since $\sigma_{\overline g}$ is a permutation, the $\overline\imath$-th row is a permutation of the first row.
\end{proof}

In the following, it is helpful to split $F$ in its positive and negative parts, a decomposition that Section~\ref{Linear} needs in order to handle non-expansiveness. That is, let us consider the linear maps ${F^\oplus},{F^\ominus}:\R^X\to\R^Y$ defined by setting ${F^\oplus}(\mathds{1}_{x_j}) := \sum_{i=1}^{m}\max \left\{b_{ij},0\right\} \mathds{1}_{y_i}$ and ${F^\ominus}(\mathds{1}_{x_j}) := \sum_{i=1}^{m}\max \left\{-b_{ij},0\right\} \mathds{1}_{y_i}$ for every index $j\in \underline{n}$. By applying Lemma~\ref{rijhm}, one can easily check that the following properties hold:

\begin{enumerate}
  \item ${F^\oplus},{F^\ominus}$ are $T$-equivariant linear maps;
  \item The matrices $B^\oplus$ and $B^\ominus$ associated with ${F^\oplus}$ and ${F^\ominus}$ with respect to the bases $\mathcal{B}_X$ and $\mathcal{B}_Y$ are $B^\oplus =\left(b^\oplus_{ij}\right) = \left( \max\left\{b_{ij},0\right\}\right)$ and $B^\ominus = \left(b^\ominus_{ij}\right)=\left(\max\left\{-b_{ij},0\right\}\right)$, respectively. In particular, $B^\oplus$ and $B^\ominus$ are non-negative matrices;
  \item $b_{ij}^\oplus \neq 0 \implies b_{ij}^\ominus = 0$, and $b_{ij}^\ominus \neq 0 \implies b_{ij}^\oplus = 0$ for any pair of indices $i,j$;
  \item $F=F^\oplus - F^\ominus$ and $B={B^\oplus}-{B^\ominus}$;
    \item \label{property5} Lemma~\ref{lempermutationrow} and the definitions of $B^\oplus, B^\ominus$ imply that two $n$-tuples of positive real numbers $\beta^\oplus=\left(\beta^\oplus_1,\dots,\beta^\oplus_{n}\right)$, $\beta^\ominus=\left(\beta^\ominus_1,\dots,\beta^\ominus_{n}\right)$ exist such that each row of ${B^\oplus}$ can be obtained by permuting $\beta^\oplus$, and each row of ${B^\ominus}$ can be obtained by permuting $\beta^\ominus$.
\end{enumerate}

The next step in order to prove Theorem~\ref{maint} is to express $F^\oplus$ and $F^\ominus$ as weighted sums of $\varphi h$, for $h \in X^Y$. First, we need to establish a connection between the elements of $X^Y$ and the rectangular permutation matrices. Each function $h \colon Y \to X$ can be associated with a $m \times n$ rectangular permutation matrix $R(h)=(r_{ij}(h))$ defined by setting $r_{ij}(h)=1$ if $h(y_i)=x_j$, and $r_{ij}(h)=0$ otherwise. In the case $X = Y$ and $h \colon X \to X$ is a permutation, we denote as $P(h)$ such square matrix, which is the usual permutation matrix. Hence, there is a bijection between $X^Y$ and the set $\mathcal{RP}_{m \times n}$
of all $m \times n$ rectangular permutation matrices, given by $h \mapsto R(h)$.

\begin{rem}\label{grhughaepj}
Assume that $h$ is a function from $Y$ to $X$ and $\mathcal{R}_h$ denotes the linear operator that sends each function $\varphi$ in $\R^X$ to $\varphi h$ in $\R^Y$. One could easily check that $R(h)$ is the matrix associated with the operator $\mathcal{R}_h$ with respect to the bases $\mathcal{B}_X$ for $\R^X$ and $\mathcal{B}_Y$ for $\R^Y$. 
Moreover, considering two functions ${f_1} \colon X \to Y$, ${f_2} \colon Y \to Z$, and a function $\p\in \R^Z$, we know that $\mathcal{R}_{{f_2} {f_1}}(\varphi)= \varphi ({f_2} {f_1}) = (\varphi {f_2}) {f_1} = \mathcal{R}_{f_1} \mathcal{R}_{f_2}(\varphi)$. Hence, $ R({f_2}{f_1})=R({f_1}) R({f_2})$.
\end{rem}

\begin{proposition}
\label{prop:sum}
    For every $h \in X^Y$ there exist two non-negative real numbers $c^\oplus(h), c^\ominus(h)$ such that
    \begin{align*}
    F^\oplus(\varphi) & = \sum_{h \in X^Y}c^\oplus(h) \varphi h, \\
    F^\ominus(\varphi) & = \sum_{h \in X^Y}c^\ominus(h) \varphi h,
    \end{align*}
    for every $\varphi \in \R^X$.
    Therefore  
\begin{align*}
    B^\oplus & = \sum_{h \in X^Y} c^\oplus(h)\, R(h), \\
    B^\ominus & = \sum_{h \in X^Y} c^\ominus(h)\, R(h).
\end{align*}
\end{proposition}

\begin{proof}
    Let us start by considering the statement concerning $c^\oplus$ and $F^\oplus$. If $F^\oplus$ is the zero operator, the statement follows immediately. Otherwise, applying Lemma~\ref{lempermutationrow} and setting $\lVert \beta^\oplus\rVert_1:= \sum_{i=1}^n \beta^\oplus_i\neq 0$, we note that the matrix $\frac{1}{\lVert \beta^\oplus\rVert_1} B^\oplus$ is stochastic. Hence, Remark~\ref{grhughaepj} and Theorem~\ref{thmrectangular} imply that
    \begin{align*}
    [F^\oplus(\varphi)] & = \lVert \beta^\oplus\rVert_1 \left( \frac{1}{\lVert \beta^\oplus\rVert_1}B^\oplus [\varphi]\right) \\
    & = \lVert \beta^\oplus\rVert_1 \sum_{h \in X^Y}\gamma^\oplus (h)R(h)[\varphi] \\
    & = \sum_{h \in X^Y} c^\oplus(h) [\mathcal{R}_h (\varphi)] \\
    & = \left[\sum_{h \in X^Y} c^\oplus(h) \varphi h\right],
    \end{align*}
    where $\gamma^\oplus(h)$ is a nonnegative real number, $\sum_{h \in X^Y}\gamma^\oplus(h) = 1$, and $c^\oplus(h)= \lVert\beta^\oplus\rVert_1 \gamma^\oplus(h)\ge 0$ for any $h \in X^Y$.
    Therefore, $F^\oplus(\varphi)=\sum_{h \in X^Y} c^\oplus(h) \varphi h$.
    Moreover, it follows that
    \begin{align*}
    B^\oplus & = \lVert \beta^\oplus\rVert_1 \left( \frac{1}{\lVert \beta^\oplus\rVert_1}B^\oplus \right) \\
    & = \lVert \beta^\oplus\rVert_1 \sum_{h \in X^Y}\gamma^\oplus (h)R(h) \\
    & = \sum_{h \in X^Y} c^\oplus(h) R(h).
    \end{align*}
    The proof of the statements concerning $c^\ominus$, $F^\ominus$, and $B^\ominus$ is analogous.
\end{proof}

\begin{rem}\label{gammanotunique}
It is important to observe that, due to Remark~\ref{notunique}, the choice of the
coefficients $\gamma^\oplus(h)$ and $\gamma^\ominus(h)$ in the proof of
Proposition~\ref{prop:sum} is not unique in general.
\end{rem}

We recall that our aim is to associate a generalized $T$-permutant measure to a
GEO $(F,T)$, and we split this task into its positive and negative parts
$F^\oplus$ and $F^\ominus$. At this stage, the function $h \mapsto c^\oplus(h)$
is not a generalized $T$-permutant measure. In order to obtain a generalized
$T$-permutant measure, we need to average it over the orbits of $h$ under the
action of $G$. First, we show that this averaging procedure is well defined.

\begin{proposition}\label{propc}
If $h_1, h_2 \in X^Y$ and there exists $t \in \underline{m}$ such that $h_1(y_t) = h_2(y_t)$, then either $c^\oplus(h_1)=0$, or $c^\ominus(h_2)=0$ or both values are null.
\end{proposition}

\begin{proof}
    Since there exist $t \in \underline{m}$ and $s\in\underline{n}$ such that $x_s=h_1(y_t) = h_2(y_t)$, we can consider $\mathds{1}_{x_s}$, and get that
    \begin{align*}
        [F^\oplus \left( \mathds{1}_{x_s} \right)] & = B^\oplus [\mathds{1}_{x_s}] \\
        & = \left[\sum_{i = 1}^m b^\oplus_{is}\mathds{1}_{y_i}\right],
    \end{align*}
    i.e., $F^\oplus \left( \mathds{1}_{x_s} \right)=\sum_{i = 1}^m b^\oplus_{is}\mathds{1}_{y_i}$.
    Hence, we have that
\begin{align*}
        b^\oplus_{ts} & = \left( \sum_{i = 1}^m b^\oplus_{is}\mathds{1}_{y_i}\right)(y_t)
        \\ & = F^\oplus \left( \mathds{1}_{x_s} \right)(y_t)
        \\& = \sum_{h \in X^Y}c^\oplus(h) \mathds{1}_{x_s}h(y_t)
        \\ & \ge c^\oplus(h_1) \mathds{1}_{x_s}h_1(y_t)
        \\ & = c^\oplus(h_1) \mathds{1}_{x_s}(x_s)
        \\ & = c^\oplus(h_1).
    \end{align*}
    One could similarly check that $b^\ominus_{ts} \ge c^\ominus({h_2})$. Therefore,
    \[
    c^\oplus({h_1})>0 \implies b^\oplus_{ts}>0 \implies b^\ominus_{ts}=0 \implies c^\ominus({h_2})=0.
    \]
    It follows that either $c^\oplus(h_1)=0$, or $c^\ominus(h_2)=0$, or both.
\end{proof}

\begin{corollary}
\label{corc}
For every $h\in X^Y$, either $c^\oplus(h)=0$, or $c^\ominus(h)=0$, or both.
\end{corollary}

\begin{proof}
    Set $h_1=h_2$ in Proposition~\ref{propc}.
\end{proof}

We finally have the prerequisites to define the generalized $T$-permutant measures $\mu^\oplus$ and $\mu^\ominus$ associated with $F^\oplus$ and $F^\ominus$. Each of them is the average of the functions $c^\oplus$ and $c^\ominus$ along the orbit $\mathcal{O}(h)$ of $h$ under the action $\alpha_T$. Formally, we define
\begin{align*}
    \mu^\oplus(h) & := \sum_{f \in \mathcal{O}(h)} \frac{c^\oplus(f)}{\left|\mathcal{O}(f)\right|} = \sum_{f \in \mathcal{O}(h)} \frac{c^\oplus(f)}{\left|\mathcal{O}(h)\right|}, \\
    \mu^\ominus(h) & := \sum_{f \in \mathcal{O}(h)} \frac{c^\ominus(f)}{\left|\mathcal{O}(f)\right|} = \sum_{f \in \mathcal{O}(h)} \frac{c^\ominus(f)}{\left|\mathcal{O}(h)\right|}.
\end{align*}

\begin{proposition}
\label{prop:gpm}
    $\mu^\oplus$ and $\mu^\ominus$ are generalized $T$-permutant measures. As a consequence, the function $\mu = \mu^\oplus - \mu^\ominus$ is also a generalized $T$-permutant measure.
\end{proposition}

\begin{proof}
    The definition of $\mu^\oplus$ immediately implies that $\mu^\oplus(H) = \mu^\oplus\left(gHT(g)^{-1}\right)$ for every $g \in G$ and every subset $H$ of $X^Y$. In other words, $\mu^\oplus$ is a non-negative generalized $T$-permutant measure. Quite analogously, we can prove that $\mu^\ominus$ is a non-negative generalized $T$-permutant measure. Hence, we can conclude that the function $\mu := \mu^\oplus - \mu^\ominus$ is a generalized $T$-permutant measure.
\end{proof}

Now that we have defined the two generalized $T$-permutant measures
$\mu^\oplus$ and $\mu^\ominus$ associated with the $T$-equivariant linear maps
$F^\oplus$ and $F^\ominus$, we can begin the steps needed to show that suitable
weighted sums of these measures reproduce precisely the two maps. 
Let $G_h:=\{g \in G \mid \alpha_T(g,h)=h\}$ be the stabilizer subgroup of $G$ with respect to $h$, i.e., the subgroup of $G$ containing the elements that fix $h$ by the action. We recall that acting on $h$ by every element of $G$ we obtain each element of the orbit $\mathcal{O}(h)$ exactly $\left|G_h\right|$ times, and the well known relation $\left|G_h\right|\left|\mathcal{O}(h)\right| = \left|G\right|$ (cf. \cite{aschbacher_2000}). We observe that, for $f, h \in X^Y$, if $f \in \mathcal{O}(h)$ then $G_f$ is isomorphic to $G_h$. Finally we are able to prove Proposition~\ref{carperm+-}, which states something very similar to Theorem~\ref{maint} but regarding $F^\oplus$ and $F^\ominus$.

\begin{proposition} 
    \label{carperm+-}
    For any $\varphi \in \R^X$ we have that
    \begin{align*}
        F^\oplus(\varphi) & = \sum_{h \in X^Y} \varphi h \ \mu^\oplus(h), \\
        F^\ominus(\varphi) & = \sum_{h \in X^Y} \varphi h \ \mu^\ominus(h).
    \end{align*}
Therefore,
 \begin{align*}
        B^\oplus & = \sum_{h \in X^Y} \mu^\oplus(h) R(h), \\
        B^\ominus & = \sum_{h \in X^Y} \mu^\ominus(h)R(h).
    \end{align*}
\end{proposition}

\begin{proof}
    Recalling that $\mathcal{R}_g \colon \varphi \mapsto \varphi g$ and $\mathcal{R}_k \colon \psi \mapsto \psi k $ are linear maps for any $g \in G$ and $k \in K$, the $T$-equivariance condition $F^\oplus \mathcal{R}_{g^{-1}} = \mathcal{R}_{T(g^{-1})}F^\oplus$ directly implies that $B^\oplus P(g^{-1}) = P\left(T(g^{-1})\right)B^\oplus$. In particular, we have that $P\left(T(g)\right)B^\oplus P(g)^{-1} = B^\oplus$ for every $g \in G$. Note that $P(g)$ is an $n\times n$ matrix, while $P(T(g))$ is an $m\times m$ matrix. From Proposition~\ref{prop:sum} it follows that
    \begin{align*}
        B^\oplus & = \overbrace{\frac{1}{\left|G\right|}B^\oplus + \dots + \frac{1}{\left|G\right|}B^\oplus}^{\left|G\right| \text{ summands}} \\
        & = \frac{1}{\left|G\right|}\sum_{g \in G}P(T(g))B^\oplus P(g)^{-1} \\
        & = \frac{1}{\left|G\right|}\sum_{g \in G}P(T(g))\left(\sum_{f \in X^Y} c^\oplus(f)R(f)\right)P(g)^{-1} \\
        & = \sum_{f \in X^Y}\sum_{g \in G}\frac{c^\oplus(f)}{\left|G\right|} P(T(g))R(f)P(g)^{-1} \\
        & = \sum_{f \in X^Y}\frac{c^\oplus(f)}{\left|G\right|}\sum_{g \in G} R\left(g^{-1}fT(g)\right) \tag{$\star$} \\
        & = \sum_{f \in X^Y}\frac{c^\oplus(f)}{\left|G\right|}\sum_{g \in G} R\left(gfT(g)^{-1}\right),
    \end{align*}
    where $(\star)$ follows from Remark~\ref{grhughaepj}, recalling that $P(g)=R(g)$ and $P(T(g))=R(T(g))$. Therefore, we have that
    \begin{align}
        [F^\oplus(\varphi)] & = B^\oplus[\p] \nonumber\\
        & = \sum_{f \in X^Y} \frac{c^\oplus(f)}{\left|G\right|} \sum_{g \in G} R\left(gfT(g)^{-1}\right)[\varphi] \nonumber \\
        & = \left[\sum_{f \in X^Y} \frac{c^\oplus(f)}{\left|G\right|} \sum_{g \in G} \mathcal{R}_{gfT(g)^{-1}}(\varphi)\right] \nonumber\\
        & = \left[\sum_{f \in X^Y} \frac{c^\oplus(f)}{\left|G\right|} \sum_{g \in G} \varphi g f T(g)^{-1}\right],\nonumber
    \end{align}
    i.e., 
    \begin{align}
    \label{eq:conj}
        F^\oplus(\varphi) & = \sum_{f \in X^Y} \frac{c^\oplus(f)}{\left|G\right|} \sum_{g \in G} \varphi g f T(g)^{-1}.
    \end{align}
    Let us now set $\delta(f_1, f_2) = 1$ if $f_1$ and $f_2$ belong to the same orbit under the action of $G$, and $\delta(f_1, f_2) = 0$ otherwise. Moreover, we recall that, by acting on $h$ with every element of $G$, each element of the orbit $\mathcal{O}(h)$ is obtained exactly $|G_h| = \frac{|G|}{|\mathcal{O}(h)|}$ times, and that $f \in \mathcal{O}(h) \iff h \in \mathcal{O}(f)$. Therefore, Equation~(\ref{eq:conj}) implies that
    \begin{align*}
        F^\oplus(\varphi) & = \sum_{f \in X^Y} \frac{c^\oplus(f)}{\left|G\right|} \left|G_f\right| \sum_{h \in \mathcal{O}(f)} \varphi h \\
        & = \sum_{f \in X^Y} \frac{c^\oplus(f)}{\left|G\right|} \left|G_f\right| \sum_{h \in X^Y} \delta(h, f)\varphi h \\
        & = \sum_{h \in X^Y} \left( \sum_{f \in X^Y} \frac{c^\oplus(f)}{\left|G\right|} \left|G_f\right| \delta(h, f) \right) \varphi h \\
        & = \sum_{h \in X^Y} \left( \sum_{f \in X^Y} \frac{c^\oplus(f)}{\left|\mathcal{O}(f)\right|} \delta(h, f) \right) \varphi h \\
        & = \sum_{h \in X^Y} \left( \sum_{f \in \mathcal{O}(h)} \frac{c^\oplus(f)}{\left|\mathcal{O}(f)\right|} \right) \varphi h \\
        & = \sum_{h \in X^Y} \varphi h \ \mu^\oplus(h).
    \end{align*}
    The equality $B^\oplus = \sum_{h \in X^Y} \mu^\oplus(h) R(h)$ follows by observing that 
    $[F^\oplus(\p)]=B^\oplus [\p]$ and
    $\left[\varphi h \ \mu^\oplus(h)\right]=\mu^\oplus(h) R(h)[\p]$ for every $h\in X^Y$ and $\p\in \R^X$. The statements concerning $F^\ominus$ and $B^\ominus$ are proved analogously.
\end{proof}

We can now prove Theorem~\ref{maint}.
\begin{proof}[Proof of Theorem~\ref{maint}]
   Proposition~\ref{prop:GEOperm} already provides the ``if'' part of the theorem. From Proposition~\ref{prop:gpm} and Proposition~\ref{carperm+-}, it follows that $\mu$ is a generalized $T$-permutant measure such that $F(\varphi) = \sum_{h \in X^Y} \varphi h \ \mu(h)$ for every $\varphi \in \mathbb{R}^X$.
\end{proof}

\begin{rem}\label{measurenotunique}
Because of Remark~\ref{gammanotunique}, the measures $\mu^\oplus$, $\mu^\ominus$ and $\mu$ associated with $F$ in the proof of Theorem~\ref{maint} are not unique in general.
\end{rem}

We wish to emphasize here that the statement of our representation theorem for
linear GEOs (Theorem~\ref{maint}) requires that $T(G)$ acts transitively on $Y$. Before extending this result to include non-expansiveness,
we make an additional remark on the relationship between $\mu^\oplus$ and $\mu^\ominus$.

\begin{proposition}
    \label{prop:singular}
    The measures $\mu^\oplus$ and $\mu^\ominus$ associated with $F$ in the proof of Theorem~\ref{maint} are mutually singular.
\end{proposition}

\begin{proof}
For every $h \in X^{Y}$, we denote by $r_{ij}(h)$ the entry in position $(i,j)$ of the matrix $R(h)$. Suppose, for a contradiction, that $\mu^\oplus$ and $\mu^\ominus$ are not mutually singular, i.e., there exists a function $\bar h \in X^{Y}$ such that $\mu^{\oplus}(\bar h) > 0$ and $\mu^{\ominus}(\bar h) > 0$. Then, by the definition of $\mu^{\oplus}$ and the inequality $\mu^{\oplus}(\bar h) > 0$, there exists $f \in \mathcal{O}(\bar h)$ such that $c^{\oplus}(f) > 0$. Since $R(f)$ is a rectangular permutation matrix, for every $i \in \underline{m}$ there exists exactly one index $j_i \in \underline{n}$ such that $r_{i j_i}(f) = 1$. Hence, by Proposition~\ref{prop:sum}, it holds that
\[
b^{\oplus}_{i j_i}
= \sum_{h \in X^{Y}} c^{\oplus}(h)\, r_{i j_i}(h)
\ge c^{\oplus}(f)\, r_{i j_i}(f)
= c^{\oplus}(f) > 0.
\]
We recall that the positivity of $b^\oplus_{ij}$ implies that $b^\ominus_{ij} = 0$ for any $(i, j)\in \underline{m}\times\underline{n}$, by definition of $B^\oplus$ and $B^\ominus$. Therefore, $b^{\ominus}_{i j_i} = 0$.
On the other hand, by Proposition~\ref{carperm+-} we have
\[
b^{\ominus}_{i j_i}
= \sum_{h \in X^{Y}} r_{i j_i}(h)\ \mu^{\ominus}(h)
\ge r_{i j_i}(f)\ \mu^{\ominus}(f)
= \mu^{\ominus}(f).
\]
Since $\mu^\ominus$ is a generalized $T$-permutant measure (Proposition~\ref{prop:gpm}), it follows that $\mu^{\ominus}(f) = \mu^{\ominus}(\bar h)$. As a consequence, $b^{\ominus}_{i j_i} \ge \mu^{\ominus}(\bar h) > 0$, which contradicts the equality $b^{\ominus}_{i j_i} = 0$.
\end{proof}

Proposition~\ref{prop:singular} ensures that $\mu^\oplus$ and $\mu^\ominus$ are the Hahn-Jordan decomposition of $\mu$.

\section{Linear GENEOs and \texorpdfstring{$T$}{T}-permutant measures}
\label{Linear}
\sectionmark{Building linear GENEOs}

We now extend Theorem~\ref{maint}
to include non-expansiveness.
To this end, one
final step is required, in the form of Lemma~\ref{lemma:ne}. This result
allows us to connect the generalized $T$-permutant measure $\mu$ with the norms of $F$
and $\p$ under the same hypotheses as in Theorem~\ref{maint}.

\begin{lemma}
    \label{lemma:ne}
    Assume that $G \subseteq \mathrm{Aut}(X), K \subseteq \mathrm{Aut}(Y)$, $T \colon G \to K$ is a homomorphism such that $T(G)$ transitively acts on $Y$, and $F$ is a $T$-equivariant map from $\mathbb{R}^X$ to $\mathbb{R}^Y$.  
    Let $\mu$ be the measure associated with $F$ in the proof of Theorem~\ref{maint}.
    Then, it holds that 
    \[
    \sum_{h \in X^Y} \lvert \mu (h) \rvert = \max_{\varphi \in \R^X \setminus \{\mathbf{0}\}}\frac{\left\| F(\varphi)\right\|_\infty}{\left\| \varphi\right\|_\infty}.
    \]
\end{lemma}

\begin{proof}
The statement is trivially true if $F$ is the zero operator, since in this case
$\mu$ coincides with the zero measure on $X^Y$, because
Proposition~\ref{prop:sum} shows that $c^\oplus(h)$ and $c^\ominus(h)$ vanish for
every $h \in X^Y$. Hence, we may assume that $F$ is not the zero operator and
that $B$ is not the zero matrix.
Setting $c:= c^\oplus - c^\ominus$, Corollary~\ref{corc} implies that $\lvert c(h) \rvert = c^\oplus(h) + c^\ominus(h)$ for every $h \in X^Y$. By definition of $\mu^\oplus$ and $\mu^\ominus$, we have that, for any $h \in X^Y$:
    \begin{align*}
    \sum_{f \in \mathcal{O}(h)} \mu^\oplus(f) & = \sum_{f \in \mathcal{O}(h)} c^\oplus(f)
    , \\
    \sum_{f \in \mathcal{O}(h)} \mu^\ominus(f) & = \sum_{f \in \mathcal{O}(h)} c^\ominus(f).
    \end{align*}
    It follows that, for each $h \in X^Y$,
    \begin{align*}
    \sum_{f \in \mathcal{O}(h)} \left| \mu(f) \right| & = \sum_{f \in \mathcal{O}(h)} \mu^\oplus(f) + \sum_{f \in \mathcal{O}(h)} \mu^\ominus(f) \\
    & = \sum_{f \in \mathcal{O}(h)} c^\oplus(f) + \sum_{f \in \mathcal{O}(h)} c^\ominus(f) \\
    & = \sum_{f \in \mathcal{O}(h)} \left| c(f) \right|,
    \end{align*}
    and hence,
    \begin{equation}
    \label{eq:negeq}
    \sum_{h \in X^Y} \left| c(h) \right| = \sum_{h \in X^Y} \left| \mu(h) \right|.
    \end{equation}
    Let us set $\mathds{1}_X := \sum_{r=1}^n \mathds{1}_{x_r}$ and $\mathds{1}_Y := \sum_{s=1}^m \mathds{1}_{y_s}$. 
    By observing that 
    $R(h) [\mathds{1}_X]=[\mathds{1}_Y]$ for every $h\in X^Y$, from Proposition~\ref{prop:sum} we obtain that 
    \begin{align*}
[F^\oplus\left(\mathds{1}_X\right)]&=\sum_{h \in X^Y} c^\oplus(h) R(h) [\mathds{1}_X] = \left(\sum_{h \in X^Y}c^\oplus(h)\right) [\mathds{1}_Y],
    \\ [F^\ominus\left(\mathds{1}_X\right)]&=\sum_{h \in X^Y} c^\ominus(h) R(h) [\mathds{1}_X] = \left(\sum_{h \in X^Y}c^\ominus(h)\right) [\mathds{1}_Y].
    \end{align*}
    Since $F^\oplus$ (resp. $F^\ominus$) is a $T$-equivariant linear map, any row of $B^\oplus$ (resp. $B^\ominus$) is a permutation of the first row (Lemma~\ref{lempermutationrow}). Then, we get
    \begin{align*}
    [F^\oplus(\mathds{1}_X)] & = B^\oplus [\mathds{1}_X] = \left( \sum_{j=1}^n b^\oplus_{1j}\right) [\mathds{1}_Y]
    , \\
    [F^\ominus(\mathds{1}_X)] & = B^\ominus [\mathds{1}_X] = \left( \sum_{j=1}^n b^\ominus_{1j}\right) [\mathds{1}_Y].
    \end{align*}
    It follows that
    \begin{align*}
    \sum_{h\in X^Y}c^\oplus(h) & = \sum_{j=1}^n b^\oplus_{1j}
    , \\
    \sum_{h\in X^Y}c^\ominus(h) & = \sum_{j=1}^n b^\ominus_{1j}.
    \end{align*}
Therefore,
    \begin{align}
    \label{eq:neeq}
    \sum_{h \in X^Y} \left| c(h) \right| & = \sum_{h \in X^Y} c^\oplus(h) + \sum _{h \in X^Y} c^\ominus(h) \nonumber \\
    & = \sum_{j=1}^n b^\oplus_{1j} + \sum_{j=1}^n b^\ominus_{1j} \nonumber \\
    & = \sum_{j=1}^n \left| b_{1j} \right|.
    \end{align}
Moreover, recalling once again that any row of $B$ is a permutation of the first
row of $B$, we obtain that 
    \[
    \sum_{j=1}^n \left| b_{1j} \right| = \left| \sum_{j=1}^n b_{1j} \mathrm{sgn}(b_{1j}) \right| = \max_{i\in\underline{m}} \left| \sum_{j=1}^n b_{ij} \mathrm{sgn}(b_{1j}) \right|.
    \]
    Let us now consider $\overline{\varphi} = \sum^n_{j=1}\mathrm{sgn}(b_{1j})\mathds{1}_{x_j} \in \R^X$. Since \(B\) is not the zero matrix, it follows that $\bar \p \neq\mathbf{0}$ and $\lVert \overline{\varphi} \rVert_\infty = 1$. Moreover, we have
    \begin{align*}
        [F (\overline{\varphi})] & = B [\overline{\varphi}] \\
        & = B \sum^n_{j=1}\mathrm{sgn}(b_{1j})[\mathds{1}_{x_j}] \\
        & = \sum^n_{j=1}\mathrm{sgn}(b_{1j})B[\mathds{1}_{x_j}] \\
        & = \sum^n_{j=1}\mathrm{sgn}(b_{1j})\sum^m_{i=1}b_{ij}[\mathds{1}_{y_i}] \\
        & = \sum^m_{i=1}\left(\sum^n_{j=1}\mathrm{sgn}(b_{1j})b_{ij}\right)[\mathds{1}_{y_i}].
    \end{align*} 
    Hence, from Equations (\ref{eq:negeq}) and (\ref{eq:neeq}) we have 
    \begin{align*}
    \left\| F (\overline{\varphi}) \right\|_\infty & = \left\| \sum^m_{i=1}\left(\sum^n_{j=1}\mathrm{sgn}(b_{1j})b_{ij}\right)\mathds{1}_{y_i} \right\|_\infty \\
    & = \left\|\left(\sum^n_{j=1}\left| b_{1j} \right| \right)\mathds{1}_{Y}\right\|_\infty\\
    & = \sum^n_{j=1}\left| b_{1j} \right| \\
    & = \sum_{h \in X^Y} \left| c(h) \right| = \sum_{h \in X^Y} \left| \mu(h) \right|.
    \end{align*}
It follows that 
$\frac{ \left\| F (\overline{\varphi}) \right\|_\infty}{\left\| \overline{\varphi} \right\|_\infty}= \left\| F (\overline{\varphi}) \right\|_\infty = \sum_{h \in X^Y} \left| \mu(h) \right|$. In particular,
    \[
    \max_{\varphi\in\R^X\setminus\{\mathbf{0}\}}\frac{\left\| F(\varphi)\right\|_\infty}{\left\|\varphi\right\|_\infty} \ge \frac{ \left\| F (\overline{\varphi}) \right\|_\infty}{\left\| \overline{\varphi} \right\|_\infty}= \sum_{h \in X^Y} \left| \mu(h) \right|.
    \]
    Moreover, from Theorem~\ref{maint} we have that, for every $\varphi\in\R^X, F(\varphi)=\sum_{h\in X^Y}\varphi h\ {\mu}(h)$. Hence,
    \begin{align*}
        \left\| F(\varphi) \right\|_\infty & \le \sum_{h\in X^Y}\left\|\varphi h\right\|_\infty \left|{\mu}(h)\right| \\
        & \le \left\|\varphi\right\|_\infty\sum_{h\in X^Y} \left|{\mu}(h)\right|.
    \end{align*}
    Therefore, $\frac{\left\| F(\varphi)\right\|_\infty}{\left\|\varphi\right\|_\infty}\le\sum_{h\in X^Y}{\left|\mu(h)\right|}$ for every $\varphi\in\R^X\setminus\{\mathbf{0}\}$. In particular, the same applies for the maximum over $\p \in \R^X \setminus\{\mathbf{0}\}$. In conclusion, 
    \[
    \sum_{h\in X^Y}{|\mu(h)|}=\max_{\varphi\in\R^X\setminus\{\mathbf{0}\}}\frac{\left\| F(\varphi)\right\|_\infty}{\left\|\varphi\right\|_\infty}.
    \]
\end{proof}

Finally, we are able to state and prove the main result of this section, namely the representation theorem for linear GENEOs.

\begin{theorem}
    \label{teo:linear}
    Assume that $G \subseteq \mathrm{Aut}(X), K \subseteq \mathrm{Aut}(Y)$, $T \colon G \to K$ is a group homomorphism, $T(G) \subseteq K$ transitively acts on the finite set $Y$ and $F$ is a map from $\mathbb{R}^X$ to $\mathbb{R}^Y$. The map $(F,T)$ is a linear GENEO from $(\mathbb{R}^X, G)$ to $(\mathbb{R}^Y, K)$ if and only if a generalized $T$-permutant measure $\mu$ exists such that $F(\varphi) = \sum_{h \in X^Y} \varphi h \ \mu(h)$ for every $\varphi \in \mathbb{R}^X$ and $\sum_{h \in X^Y} \left| \mu(h) \right| \le 1$.
\end{theorem}

\begin{proof}
    Let us assume that $(F,T)$ is a GENEO from $\left( \R^X, G \right)$ to $\left( \R^Y, K \right)$ with respect to $T$. Then, Theorem~\ref{maint} guarantees that a generalized $T$-permutant measure $\mu$ exists such that $F(\p) = \sum_{h \in X^Y} \p h \ \mu(h)$ for every $\p \in \R^X$. Moreover, Lemma~\ref{lemma:ne} guarantees that $\sum_{h \in X^Y} \left| \mu(h) \right| = \max_{\varphi \in \R^X \setminus \{\mathbf{0}\}} \frac{\left\| F(\p) \right\|_\infty}{\left\| \p \right\|_\infty}$. Since $F$ is non-expansive, $\sum_{h \in X^Y} \left| \mu(h) \right| \le 1$, which is the first implication of the statement. Let us now assume that a generalized $T$-permutant measure $\mu$ exists such that $F(\p) = \sum_{h \in X^Y} \p h \ \mu(h)$ for every $\p \in \R^X$, with $\sum_{h \in X^Y} \left| \mu(h) \right| \le 1$. Then, Proposition~\ref{prop:GEOperm} guarantees that $F$ is linear and $T$-equivariant. Moreover, we can prove the non-expansiveness
    of $F$:
    \begin{align*}
        \left\| F(\p) \right\|_\infty & = \left\| \sum_{h \in X^Y} \p h \ \mu(h) \right\|_\infty \\
        & \le \sum_{h \in X^Y} \left\| \p h \right\|_\infty \left| \mu(h) \right| \\
        & \le \left\| \p \right\|_\infty \left( \sum_{h \in X^Y} \left| \mu(h) \right| \right) \\
        & \le \left\| \p \right\|_\infty.
    \end{align*}
\end{proof}

Theorem~\ref{teo:linear} represents every linear GENEO between different
perception pairs by means of generalized $T$-permutant measures. As we will show in the next section, generalized permutant measures are often easy to define;
hence, this theorem provides a valuable tool for the construction of
GENEOs. 
The ability to populate the space of GENEOs with operators that can be
explicitly constructed is fundamental for obtaining a wide range of
applications. For this reason, Theorem~\ref{teo:linear} is not only important
from a theoretical point of view, but also crucial for the applicability of
GENEO theory.

\begin{rem}
    \label{rem:noinj}
    The generalized $T$-permutant measure of Theorem~\ref{teo:linear} is not unique. As an example, let us consider $X = Y = \left\{ 1, 2, 3 \right\}, G = K = \Aut(X)$ and the identity homomorphism $T = \mathrm{id}_{\Aut(X)}$. Let $F \colon \R^X \to \R^X$ be the linear application that maps $\mathds{1}_j$ to $\frac{1}{3}\sum_{i \in X} \mathds{1}_i$ for any $j \in X$. It is easy to check that $F$ is a GENEO with associated matrix
    \[ B = 
    \begin{pmatrix}
        1/3 & 1/3 & 1/3 \\
        1/3 & 1/3 & 1/3 \\
        1/3 & 1/3 & 1/3
    \end{pmatrix}.
    \]
    One could represent $B$ as:
    \[ B = 
    \begin{pmatrix}
        1/3 & 0 & 0 \\
        0 & 1/3 & 0 \\
        0 & 0 & 1/3
    \end{pmatrix} + \begin{pmatrix}
        0 & 1/3 & 0 \\
        0 & 0 & 1/3 \\
        1/3 & 0 & 0
    \end{pmatrix} +     \begin{pmatrix}
        0 & 0 & 1/3 \\
        1/3 & 0 & 0 \\
        0 & 1/3 & 0
    \end{pmatrix} =: \frac{1}{3}R_1 + \frac{1}{3}R_2 + \frac{1}{3}R_3,
    \] and
    \[ B = 
    \begin{pmatrix}
        0 & 0 & 1/3 \\
        0 & 1/3 & 0 \\
        1/3 & 0 & 0
    \end{pmatrix} + \begin{pmatrix}
        0 & 1/3 & 0 \\
        1/3 & 0 & 0 \\
        0 & 0 & 1/3
    \end{pmatrix} +     \begin{pmatrix}
        1/3 & 0 & 0 \\
        0 & 0 & 1/3 \\
        0 & 1/3 & 0
    \end{pmatrix} =: \frac{1}{3}R_4 + \frac{1}{3}R_5 + \frac{1}{3}R_6.
    \]
    One could easily check that $\left\{R_1 \right\}, \left\{R_2, R_3 \right\}$ and $\left\{R_4, R_5, R_6 \right\}$ are three orbits of $X^X$ under the action of $\alpha_T$. We can define these two different generalized $T$-permutant measures:
    \[
    \mu(h) := \begin{cases}
        1/3 & \text{if } P(h) = R_1, R_2, R_3 \\
        0 & \text{otherwise}
    \end{cases},
    \]
    \[
    \nu(h) := \begin{cases}
        1/3 & \text{if } P(h) = R_4, R_5, R_6 \\
        0 & \text{otherwise}
    \end{cases}.
    \]
    Both generalized permutant measures define the same linear GENEO $(F,\mathrm{id}_{\Aut(X)})$.
\end{rem}

\section{Compactness and convexity of the space of linear GENEOs}
\label{Compactness}

This section aims to establish several results concerning the topology and geometry
of the space of linear GENEOs, in particular compactness and
convexity.

In this section, we fix a homomorphism \(T \colon G \to K\) and let \(\mathcal{F}^T_{\mathrm{lin}}\) denote the space of all linear maps \(F \colon \mathbb{R}^X \to \mathbb{R}^Y\) such that $(F, T)$ is a GENEO. We denote by \(\Gamma\) the set of all generalized \(T\)-permutant measures and observe that \(\Gamma\) is a real vector space. 
We emphasize that, since $X^Y$ is finite, the action $\alpha_T$
admits only finitely many orbits, which we denote by
$\mathcal{O}_1, \dots, \mathcal{O}_k \subseteq X^Y$.
For each $i \in \underline{k}$, we define the signed measure
\[
\nu_i(h) :=
\begin{cases}
1 & \text{if } h \in \mathcal{O}_i, \\
0 & \text{otherwise}.
\end{cases}
\]
We observe that, for every $i\in \underline{k}$, the measure $\nu_i$ 
is a generalized $T$-permutant measure, since it is constant on
orbits. 
Moreover, the measures $\nu_1,\ldots,\nu_k$ are linearly independent. We also observe that, for every $\p \in \mathbb{R}^X$,
$\|F_{\nu_i}(\p)\|_\infty
=
\left\|
\sum_{h \in X^Y} \p h \ \nu_i(h)
\right\|_\infty
=
\left\|
\sum_{h \in \mathcal{O}_i} \p h
\right\|_\infty
\le
\sum_{h \in \mathcal{O}_i}
\|\p h\|_\infty
\le
|\mathcal{O}_i|\,\|\p\|_\infty$.

\begin{proposition}
\label{prop:convexity}
The signed measures $\nu_1, \ldots, \nu_k$ form a basis for $\Gamma$.
\end{proposition}

\begin{proof}
On the one hand, it is clear that each measure of the form $\mu=\sum_{i=1}^k a_i \ \nu_i$
is a $T$-permutant measure on $X^Y$.
On the other hand, if $\mu$ is a generalized $T$-permutant measure on $X^Y$, then for every
index $i\in\underline{k}$ there exists a value $a_i$ such that
$\mu(f)=a_i$ for every $f\in\mathcal O_i$.
Since $\nu_i$ is constant on $\mathcal O_i$, it follows that $\mu = \sum_{i=1}^k a_i \ \nu_i$.
\end{proof}

\begin{corollary}
\label{cor:convexity}
The linear GEOs from $(\R^X,G)$ to $(\R^Y,K)$ are exactly the pairs of maps  
$(\sum_{i=1}^k a_i F_{\nu_i},T)$, with $a_1,\ldots,a_k$ arbitrary real numbers. The linear GENEOs from $(\R^X,G)$ to $(\R^Y,K)$ are exactly the pairs of maps $(\sum_{i=1}^k a_i F_{\nu_i},T)$, with  $\sum_{i=1}^k |a_i||\mathcal{O}_i|\le 1$.
\end{corollary}

\begin{proof}
Theorem~\ref{maint}  guarantees that the linear GEOs from
$(\mathbb{R}^X, G)$ to $(\mathbb{R}^Y, K)$ are precisely the pairs
$(F_\mu, T)$, where $\mu$ is a generalized $T$-permutant measure on $X^Y$.
The first statement of Corollary~\ref{cor:convexity} follows immediately from
this result and Proposition~\ref{prop:convexity}, by observing that if
$\mu = \sum_{i=1}^k a_i \ \nu_i$, then
$F_\mu = \sum_{i=1}^k a_i F_{\nu_i}$.
Moreover, Theorem~\ref{teo:linear} guarantees that the linear GENEOs from
$(\mathbb{R}^X, G)$ to $(\mathbb{R}^Y, K)$ are exactly the pairs
$(F_\mu, T)$ for which $\mu$ is a $T$-permutant measure on $X^Y$
satisfying $\sum_{h \in X^Y} |\mu(h)| \le 1$. The second statement of Corollary~\ref{cor:convexity}
follows from Theorem~\ref{teo:linear} by observing that, if $\mu=\sum_{i=1}^k a_i \ \nu_i$, then 

\[
\begin{aligned}
\sum_{h\in X^Y} |\mu(h)|
&=
\sum_{h\in X^Y} \left|\sum_{j=1}^k a_j \ \nu_j(h)\right|\\
&=\sum_{i=1}^k \sum_{h\in \mathcal{O}_i} \left|\sum_{j=1}^k a_j \ \nu_j(h)\right|\\
&=\sum_{i=1}^k \sum_{h\in \mathcal{O}_i} \left|a_i \right|\\
&=\sum_{i=1}^k |a_i| |\mathcal{O}_i|.
\end{aligned}
\]
\end{proof}

We define on $\Gamma$ the  
distance $\sum_{h \in X^Y} |\mu_1(h)-\mu_2(h)|$. Moreover, we endow $\mathcal{F}_\mathrm{lin}^T$ with the usual operator norm defined as
\[
\left\| F \right\|_\mathrm{op} := \sup\left\{ \left\| F(\p) \right\|_\infty, \p \in \R^X, \left\| \p \right\|_\infty = 1 \right\}.
\]
We recall that, given a generalized $T$-permutant measure $\mu$ such that $\sum_{h\in X^Y}|\mu(h)|\le 1$, 
we can define an associated GENEO $(F_\mu,T)$ by setting 
$F_\mu(\p)=\sum_{h \in X^Y} \varphi h \ \mu(h)$.

\begin{theorem}\label{thmcomp&conv}
    Given a homomorphism $T \colon G \to K$, the space $\mathcal{F}_\mathrm{lin}^T$ is compact and coincides with the convex hull of the set $\left\{\frac{1}{|\mathcal{O}_1|}F_{\nu_1},\ldots, 
    \frac{1}{|\mathcal{O}_k|}F_{\nu_k},
    -\frac{1}{|\mathcal{O}_1|}F_{\nu_1},\ldots, 
    -\frac{1}{|\mathcal{O}_k|}F_{\nu_k}\right\}$.
\end{theorem}
\begin{proof}
Let us define the linear map $f$ that takes each $\mu\in\Gamma$ to the operator $F_\mu$. Since
$$\begin{aligned}
\left\|F_{\mu_1}(\varphi)-F_{\mu_2}(\varphi)\right\|_\infty
&=\left\|\sum_{h\in X^Y} \varphi h\ \mu_1(h)-\sum_{h\in X^Y} \varphi h\ \mu_2(h)\right\|_\infty \\
&=\left\|\sum_{h\in X^Y} \varphi h \ (\mu_1(h)-\mu_2(h))\right\|_\infty \\
&\le \sum_{h\in X^Y} \|\varphi h\|_\infty|\mu_1(h)-\mu_2(h)| \\
&\le \sum_{h\in X^Y} \|\varphi \|_\infty|\mu_1(h)-\mu_2(h)| \\
&= \|\varphi \|_\infty\sum_{h\in X^Y} |\mu_1(h)-\mu_2(h)|, 
\end{aligned}
$$
it follows that
$\left\| F_{\mu_1}-F_{\mu_2} \right\|_\mathrm{op}\le \sum_{h\in X^Y} |\mu_1(h)-\mu_2(h)|$.
This implies that the function $f$ is non-expansive, and hence continuous. Let us now set
$\Gamma' := \left\{ \mu \in \Gamma : \sum_{h \in X^Y} |\mu(h)| \le 1 \right\}$, and observe that $\mathcal{F}^T_{\mathrm{lin}} = f(\Gamma')$ (Theorem~\ref{teo:linear}).
Since $\Gamma'$ is compact, and $f$
is continuous, it follows that $\mathcal{F}^T_{\mathrm{lin}}$ is compact as well.

Corollary~\ref{cor:convexity} states that 
the linear GENEOs from $(\R^X,G)$ to $(\R^Y,K)$ are exactly the pairs of maps  
$(\sum_{i=1}^k a_i F_{\nu_i},T)$, with  $\sum_{i=1}^k |a_i|\ |\mathcal{O}_i|\le 1$. We recall that, in a real vector space, $\left\{ \sum_{i=1}^k \lambda_iv_i \mid \sum_{i=1}^k |\lambda_i|\le 1 \right\}=\mathrm{conv}
(v_1,\ldots,v_k,-v_1,\ldots,-v_k)$, where $\mathrm{conv}(S)$ is the convex hull of $S$. Therefore, the linear GENEOs from $(\R^X,G)$ to $(\R^Y,K)$ are exactly the pairs $(F,T)$ where $F$ belongs to the convex hull of the set
    $\left\{\frac{1}{|\mathcal{O}_1|}F_{\nu_1},\ldots, 
    \frac{1}{|\mathcal{O}_k|}F_{\nu_k},
    -\frac{1}{|\mathcal{O}_1|}F_{\nu_1},\ldots, 
    -\frac{1}{|\mathcal{O}_k|}F_{\nu_k}\right\}$.
\end{proof}

\begin{rem}
It is worth pointing out an alternative proof of the compactness of the space $\mathcal{F}^T_\mathrm{lin}$ in Theorem~\ref{thmcomp&conv}.
Let us consider the compact sets
$\Phi = \{\varphi \in \mathbb{R}^X : \|\varphi\|_\infty = 1\}$ and $\Psi = \{\psi \in \mathbb{R}^Y : \|\psi\|_\infty \le 1\}$.
If $F \in \mathcal{F}^T_{\mathrm{lin}}$, we observe that
$F|_{\Phi}(\Phi) \subseteq \Psi$ (since $F$ is non-expansive). Moreover, we recall the metric
$D_{\mathrm{GENEO}}(F_1|_{\Phi}, F_2|_{\Phi})
:=
\max_{\varphi \in \Phi}
\|F_1|_{\Phi}(\varphi) - F_2|_{\Phi}(\varphi)\|_\infty$ on the set $\mathcal{F}=\{F|_{\Phi}:F\in \mathcal{F}^T_\mathrm{lin}\}$
(cf., e.g., \cite{bergomi2019towards}). Of course,
$\|F_1 - F_2\|_{\mathrm{op}}
=
D_{\mathrm{GENEO}}(F_1|_{\Phi}, F_2|_{\Phi})$.
This means that the map from $(\mathcal{F}^T_\mathrm{lin},\|\cdot\|_{\mathrm{op}})$
to $(\mathcal{F},D_{\mathrm{GENEO}})$ that sends $F$ to $F|_{\Phi}$
is an isometry.
We already know that the space of GENEOs from
$(\Phi, G)$ to $(\Psi, K)$ is compact with respect to 
$D_{\mathrm{GENEO}}$ (see, e.g., \cite{bergomi2019towards}).
Therefore, the subset $\mathcal{F}$ of this space, being closed, is compact as well. Since 
$\mathcal{F}^T_\mathrm{lin}$ is isometric to 
$\mathcal{F}$, it is compact.
\end{rem}

\begin{rem}
    We stress that, in general, the set
\[
\left\{ \tfrac{1}{|\mathcal{O}_1|}F_{\nu_1}, \dots, \tfrac{1}{|\mathcal{O}_k|}F_{\nu_k},
-\tfrac{1}{|\mathcal{O}_1|}F_{\nu_1}, \dots, -\tfrac{1}{|\mathcal{O}_k|}F_{\nu_k} \right\}
\]
does not coincide with the set of vertices of the polytope $\mathcal{F}^T_{\mathrm{lin}}$.
Recall that a point $p$ of a polytope $P$ is called a vertex of $P$ if the equality
$p = \lambda p_1 + (1-\lambda)p_2$,
$\lambda \in ]0,1[,\; p_1,p_2 \in P$,
implies $p = p_1 = p_2$.
Let us now consider the setting of Remark~\ref{rem:noinj}, and the three orbits
\[
\mathcal{O}_1 := \{R_1\}, \qquad
\mathcal{O}_2 := \{R_2, R_3\}, \qquad
\mathcal{O}_3 := \{R_4, R_5, R_6\}
\]
in $X^X$, under the action $\alpha_T$.
Following the construction illustrated at the beginning of this section, we
define the corresponding measures $\nu_1, \nu_2, \nu_3$.
Accordingly, for $i \in \{1,2,3\}$, we have
$F_{\nu_i}(\varphi)
= \sum_{h \in X^Y} \varphi h\, \nu_i(h)
= \sum_{h \in \mathcal{O}_i} \varphi h$.
Since $[F_{\nu_i}(\varphi)]
= \sum_{h \in \mathcal{O}_i} R(h)[\varphi]$,
it is easy to verify that
$[F_{\nu_1}] =
\begin{pmatrix}
1 & 0 & 0 \\
0 & 1 & 0 \\
0 & 0 & 1
\end{pmatrix}$, 
$[F_{\nu_2}] =
\begin{pmatrix}
0 & 1 & 1 \\
1 & 0 & 1 \\
1 & 1 & 0
\end{pmatrix}$, 
and
$[F_{\nu_3}] =
\begin{pmatrix}
1 & 1 & 1 \\
1 & 1 & 1 \\
1 & 1 & 1
\end{pmatrix}$, 
It follows that
$\left[\frac{1}{|\mathcal{O}_1|}F_{\nu_1}\right] =
\begin{pmatrix}
1 & 0 & 0 \\
0 & 1 & 0 \\
0 & 0 & 1
\end{pmatrix}$, 
$\left[\frac{1}{|\mathcal{O}_2|}F_{\nu_2}\right] =
\begin{pmatrix}
0 & \tfrac{1}{2} & \tfrac{1}{2} \\
\tfrac{1}{2} & 0 & \tfrac{1}{2} \\
\tfrac{1}{2} & \tfrac{1}{2} & 0
\end{pmatrix}$, 
and
$\left[\frac{1}{|\mathcal{O}_3|}F_{\nu_3}\right] =
\begin{pmatrix}
\tfrac{1}{3} & \tfrac{1}{3} & \tfrac{1}{3} \\
\tfrac{1}{3} & \tfrac{1}{3} & \tfrac{1}{3} \\
\tfrac{1}{3} & \tfrac{1}{3} & \tfrac{1}{3}
\end{pmatrix}$.
In particular,
$\left[\frac{1}{|\mathcal{O}_3|}F_{\nu_3}\right] = \frac{1}{3}\left[\frac{1}{|\mathcal{O}_1|}F_{\nu_1}\right] + \frac{2}{3}\left[\frac{1}{|\mathcal{O}_2|}F_{\nu_2}\right]$.
Therefore, $F_{\nu_3}$ is not a vertex of the polytope $\mathcal{F}^T_{\mathrm{lin}}$.
\end{rem}

\section{Application}
\label{Application}
The aim of this section is to show how introducing a layer of linear GENEOs before an autoencoder increases its robustness to noise. This approach builds on prior work that leverages topological information, and specifically persistent homology, within autoencoder architectures \cite{pmlr-v119-moor20a}. Our general framework is motivated by the fact that the computation of persistence diagrams can itself be formalized as a GENEO (see Remark~\ref{rem:PH}). Specifically, we demonstrate that training an autoencoder to reconstruct MNIST digits and then testing it on noisy versions of these digits yields better performance when a layer of linear GENEOs for dimensionality reduction is introduced before the autoencoding architecture. Random toroidal translations are applied to MNIST images to enlarge the dataset, enrich the experimental setting, and exploit the translation equivariance properties of the GENEO operators. To evaluate reconstruction performance, we use a CNN trained to recognize digits on clean MNIST and track how much its performance drops when it receives the noisy MNIST reconstruction provided by the various methods.

\subsection{Mathematical setup}
Let $X = \mathbb{Z}_p \times \mathbb{Z}_p$, where $p$ is prime, and let $Y = \mathbb{Z}_p$. Throughout this section, all operations are performed modulo $p$. Let $G$ be the group of translations of $X$, that is, $G = \{g_v(q) = q + v$ for $v \in X\}$. Let $K$ be the group of translations of $Y$, that is, $K = \{k_c(z) = z + c$ for $c \in Y\}$. We consider the following perception pairs: $(\R^X, G)$ and $(\R^Y, K)$. We define the following bilinear form $\cdot : X \times X \to \mathbb{Z}_p$ given by:
\[
u \cdot v := u_1v_1 + u_2v_2,
\]
where $u = (u_1, u_2)$ and $v = (v_1, v_2)$. It is easy to check that $\cdot$ is a symmetric bilinear form on $X$.

\begin{rem}
    There exist non-trivial vectors $u \in \mathbb{Z}_p \times \mathbb{Z}_p$ (i.e., $u \neq e_i$ for any standard basis vector $e_i$) such that $u \cdot u = 1$. For example, in $\mathbb{Z}_{11}$, the vector $(3, 5)$ satisfies this property, since $(3, 5) \cdot (3, 5) = 9 + 25 = 34 \equiv 1 \pmod{11}$.
\end{rem}

Let us now fix a vector $\bar{w} = (\bar{w}_1, \bar{w}_2) \in \mathbb{Z}_p \times \mathbb{Z}_p$ such that $\bar{w} \cdot \bar{w} = 1$. We will call $\bar w$ a \textit{unit vector}, and define the group homomorphism $T_{\bar w} \colon G \to K$ by:
\[
T_{\bar w}(g_v) := k_{v\cdot \bar w}.
\]
It is easy to check that $T_{\bar w}$ is a group homomorphism. We set $\bar{w}^\perp = (-\bar{w}_2, \bar{w}_1)$. The reader should note that $\bar{w}^\perp \cdot \bar w = 0$. For every $t \in \mathbb{Z}_p$, we define $h_{\bar{w}}^t \in X^Y$,
\begin{equation}
\label{eq:permutant}
    h_{\bar{w}}^t(z) := z\bar{w} + t\bar{w}^\perp.
\end{equation}

\begin{lemma}\label{lemma:multiple}
    For any $v \in X$, the vector $v - (v \cdot \bar{w})\bar{w}$ is a scalar multiple of $\bar{w}^\perp$.
\end{lemma}

\begin{proof}
    Let $u = v - (v \cdot \bar{w})\bar{w}$. 
    First, we verify that $u \cdot \bar{w} = 0$:
    \begin{align*}
        u \cdot \bar{w} &= (v - (v \cdot \bar{w})\bar{w}) \cdot \bar{w} \\
        &= v \cdot \bar{w} - (v \cdot \bar{w})(\bar{w} \cdot \bar{w}) \\
        &= v \cdot \bar{w} - v \cdot \bar{w} = 0.
    \end{align*}
  We can thus observe that the following homogeneous square linear system in the unknowns $w_1$ and $w_2$ admits $\bar w$ as a nonzero solution:
    \[
    \begin{cases}
        u \cdot w = u_1w_1 + u_2w_2 = 0 \\
        \bar{w}^\perp \cdot w = -\bar{w}_2w_1 + \bar{w}_1w_2 = 0.
    \end{cases}
    \]
It follows directly from elementary linear algebra (cf. Theorem~7 in \cite{hoffmann1971linear}) that $u$ and $\bar{w}^\perp$ are linearly dependent.
\end{proof}

\begin{proposition}
\label{prop:permutant}
    The set $H_{\bar w} = \left\{ h_{\bar w}^t, t \in \Z_p\right\}$ is invariant with respect to the action $\alpha_{T_{\bar{w}}}$.
\end{proposition}
\begin{proof}
    For every $t\in\Z_p$ and $v\in \Z_p\times \Z_p$, we have
    \begin{align*}
        g_v \circ h_{\bar w}^t \circ T_{\bar w}(g_v^{-1})(z) &= g_v \circ h_{\bar w}^t \circ T_{\bar w}(g_{-v})(z) \\
        &= g_v\left(h_{\bar w}^t\left(k_{-v\cdot \bar w}(z)\right)\right) \\
        &= g_v\left(h_{\bar w}^t(z - v \cdot \bar w) \right) \\
        &= g_v \left( (z-v\cdot \bar w)\bar w + t \bar{w}^\perp \right) \\
        &= (z - v\cdot \bar w)\bar w + t \bar{w}^\perp + v \\
        &= z\bar w +t\bar{w}^\perp + (v - \left(v\cdot \bar w)\bar w\right) \\
        &= z\bar w + t\bar{w}^\perp + t^\prime \bar{w}^\perp \\
        &= h_{\bar w}^{t + t^\prime}(z)
    \end{align*}
  for a suitable real number $t'$, because of Lemma~\ref{lemma:multiple}.
Hence, $\alpha_{T_{\bar{w}}}(h_{\bar{w}}^{t}) \in H_{\bar w}$, which concludes the proof.
\end{proof}

Proposition~\ref{prop:permutant} enables the direct definition of a linear GENEO using Theorem~\ref{teo:linear}. Denoting with $\mu_{\bar w}$ the measure $\mu_{\bar w}(h) = 1/|H_{\bar w}|$ if $h \in H_{\bar w}$, $\mu_{\bar w}(h) = 0$ otherwise, it is straightforward to check that $\mu_{\bar w}$ is a generalized $T_{\bar{w}}$-permutant measure. Hence, the operator $F_{\mu_{\bar w}}$ that maps 
$\varphi \mapsto {\sum_{h \in X^Y} \p h \ \mu_{\bar w}(h)}=\sum_{h \in H_{\bar w}} \frac{\varphi h}{|H_{\bar w}|}$ is a linear GENEO. Intuitively, $F_{\mu_{\bar w}}(\varphi)(z)$ is the average value of $\varphi$ over the line indexed by $z$ in the family of $p$ parallel lines determined
by $\bar w^\perp$ on the torus. Figure~\ref{fig:geneo-lines} shows two
operators $F_{\mu_{\bar w}}$ applied to images of two digits.

\begin{figure}[t]
  \centering
  \includegraphics[width=\textwidth]{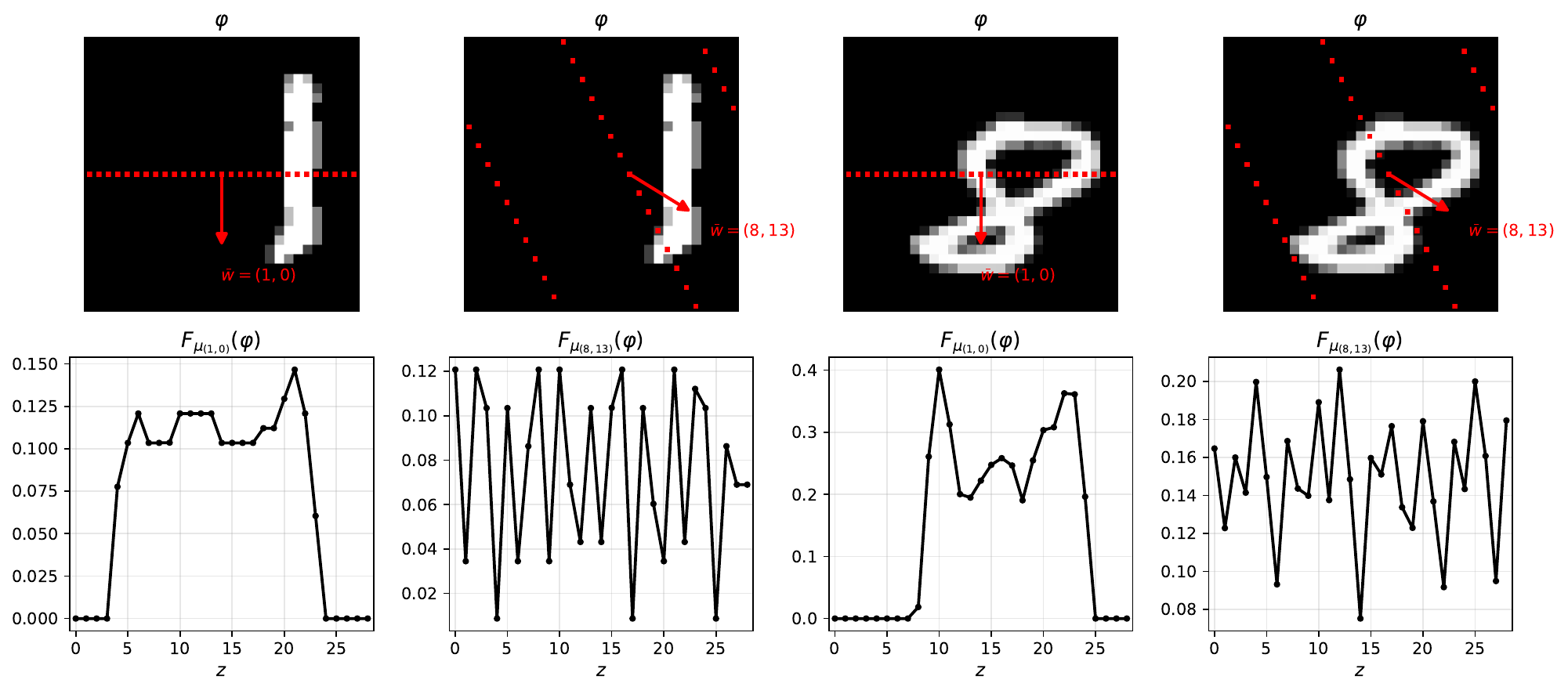}
  \caption{
  Two linear GENEOs produced by our model for
$X=\mathbb{Z}_{29}\times\mathbb{Z}_{29}$, corresponding to
$\bar{w}=(1,0)$ and $\bar{w}=(8,13)$, respectively.
  Top row: the input $\varphi$, with
    the line $\{h^t_{\bar w}(z_0) : t \in \mathbb{Z}_{29}\}$ marked in red for
    $z_0 = 14$, together with the corresponding unit vector $\bar w$; note that
    $\bar w$ and $\bar w^{\perp}$ are orthogonal with respect to the bilinear
    form modulo $p$, so the drawn vector is not Euclidean-perpendicular to the
    line. Bottom row: the output $F_{\mu_{\bar w}}(\varphi)$, whose value at $z$
    averages $\varphi$ over the line of index $z$.}
  \label{fig:geneo-lines}
\end{figure}

\subsection{Experimental setup}
To enable the application of GENEO operators, all methods use a zero-padded version of MNIST images resized to 29×29 pixels, where 29 is chosen as the closest prime number to the original dimension. The CNN classifier consists of three convolutional layers (respectively, 32, 64, and 128 filters) with ReLU activations and max pooling, followed by three fully connected layers (respectively, 256, 128, and 10 units) with dropout regularization (0.25). The network was trained on translated MNIST images ($29\times29$ pixels) for $10$ epochs using Adam optimizer with learning rate 0.001 and cross-entropy loss. On clean test data, the classifier achieved an accuracy above 98\%.

In this study, we compare four approaches: three baseline methods from standard literature (autoencoder, convolutional autoencoder, and variational autoencoder) and a GENEO-enhanced version of the convolutional autoencoder that consists in a preliminary GENEO layer. The standard autoencoder (AE, \cite{hinton2006reducing}) consists of an encoder ($841\to256\to128\to32$) and symmetric decoder with ReLU activations. The convolutional autoencoder (CAE, \cite{masci2011stacked}) uses two convolutional layers with max pooling in the encoder, a 64-dimensional bottleneck, and transposed convolutions in the decoder. The variational autoencoder (VAE, \cite{kingma2013auto}) follows the standard architecture with encoder layers ($841\to256\to128$) mapping to 32-dimensional mean and log-variance vectors, using the reparameterization trick and trained with reconstruction loss plus KL divergence. For the GENEO-based approach that we refer as GENEO CAE, we generate 28 distinct GENEOs, where 28 is the number of unit vectors in $\Z_{29}\times\Z_{29}$. Each unit vector generates a different linear GENEO using the previously defined strategy. The first GENEO-layer is formed by taking the 28-tuple of GENEO outputs, resulting in a combined $28\times29$ dimensional representation. This GENEO-transformed data serves as input to the convolutional autoencoder architecture. This architecture combines the equivariant feature extraction of GENEOs with the spatial processing capabilities of convolutional networks. All models were trained for 30 epochs with Adam optimizer (learning rate 0.001) and MSE reconstruction loss. To evaluate reconstruction robustness, we corrupt test images with salt \& pepper noise at levels ranging from 0.05 to 0.4, then measure the drop in CNN classification accuracy when the classifier receives reconstructed images from each model instead of the original clean images. The choice of the salt \& pepper noise is motivated by the degradation regime it induces. Under additive Gaussian noise, for instance, the classifier stays above $96\%$ on the raw corrupted images up to $\sigma = 0.20$, leaving autoencoders' reconstructions no margin to contribute.

\subsection{Results}

We evaluate the reconstruction performance of the four approaches on MNIST digits corrupted with salt \& pepper noise at levels ranging from 5\% to 40\%.

\begin{figure}[t]
\centering
\includegraphics[width=0.48\textwidth]{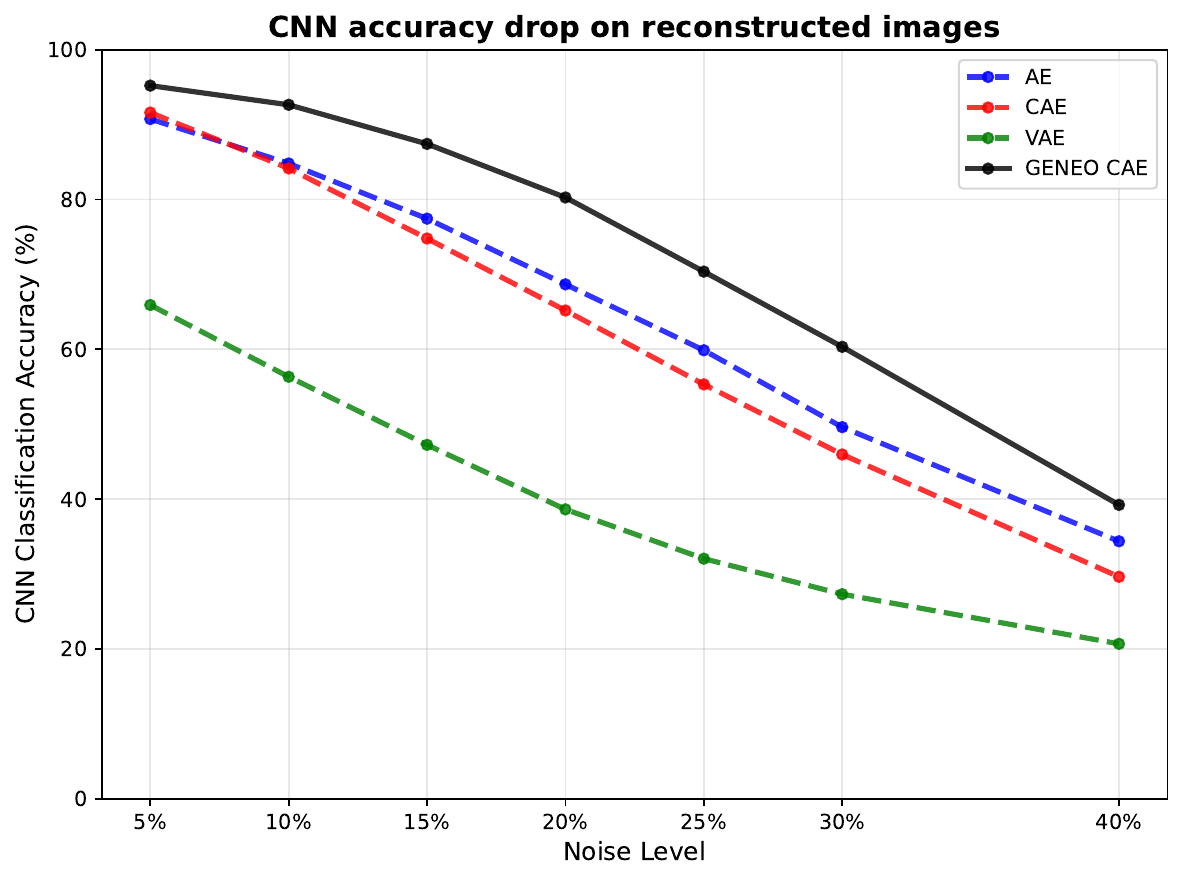}
\includegraphics[width=0.48\textwidth]{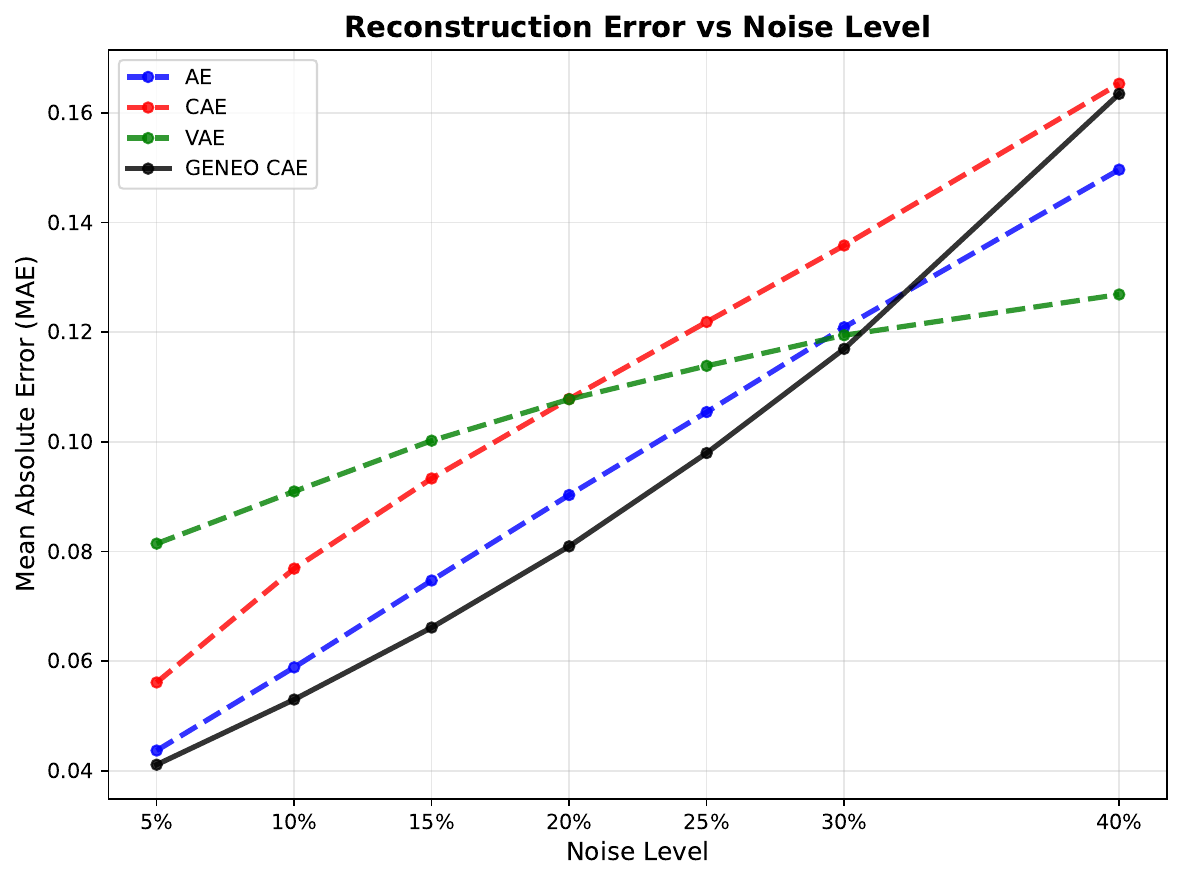}
\caption{Left: CNN classification accuracy on reconstructed images. Right: Mean Absolute Error (MAE) between reconstructed and clean images as a complementary measure. The GENEO CAE achieves both the highest classification accuracy and, with the sole exception of 40\% noise, the lowest reconstruction error.}
\label{fig:accuracy}
\end{figure}
Figure~\ref{fig:accuracy} shows the CNN classification accuracy (left) and Mean Absolute Error (right) as a complementary reconstruction measure. The GENEO CAE demonstrates superior performance, achieving the highest classification accuracy across all noise levels while also maintaining the lowest MAE in most cases. This dual improvement indicates that the equivariant features extracted by GENEO operators not only preserve semantic content for accurate classification but also enable better pixel-level reconstruction fidelity. At 15\% noise corruption, the GENEO CAE preserves approximately 90\% accuracy while both standard and convolutional autoencoders drop below 80\%, with consistently lower reconstruction error. The complementary MAE plot further confirms the reconstruction quality of the GENEO CAE across the noise spectrum.

\begin{figure}[t]
    \centering
    \includegraphics[width=0.48\textwidth]{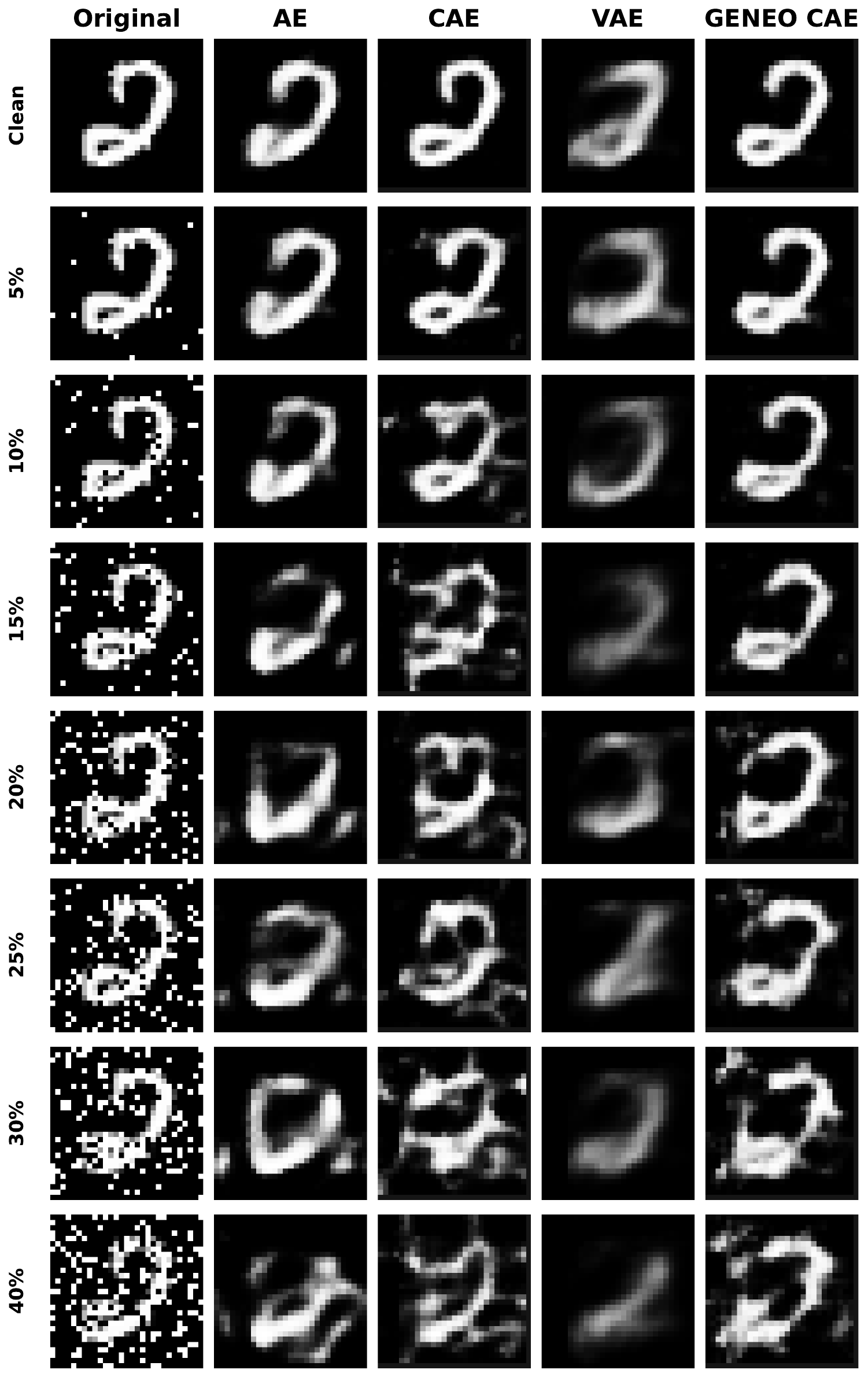}
    \includegraphics[width=0.48\textwidth]{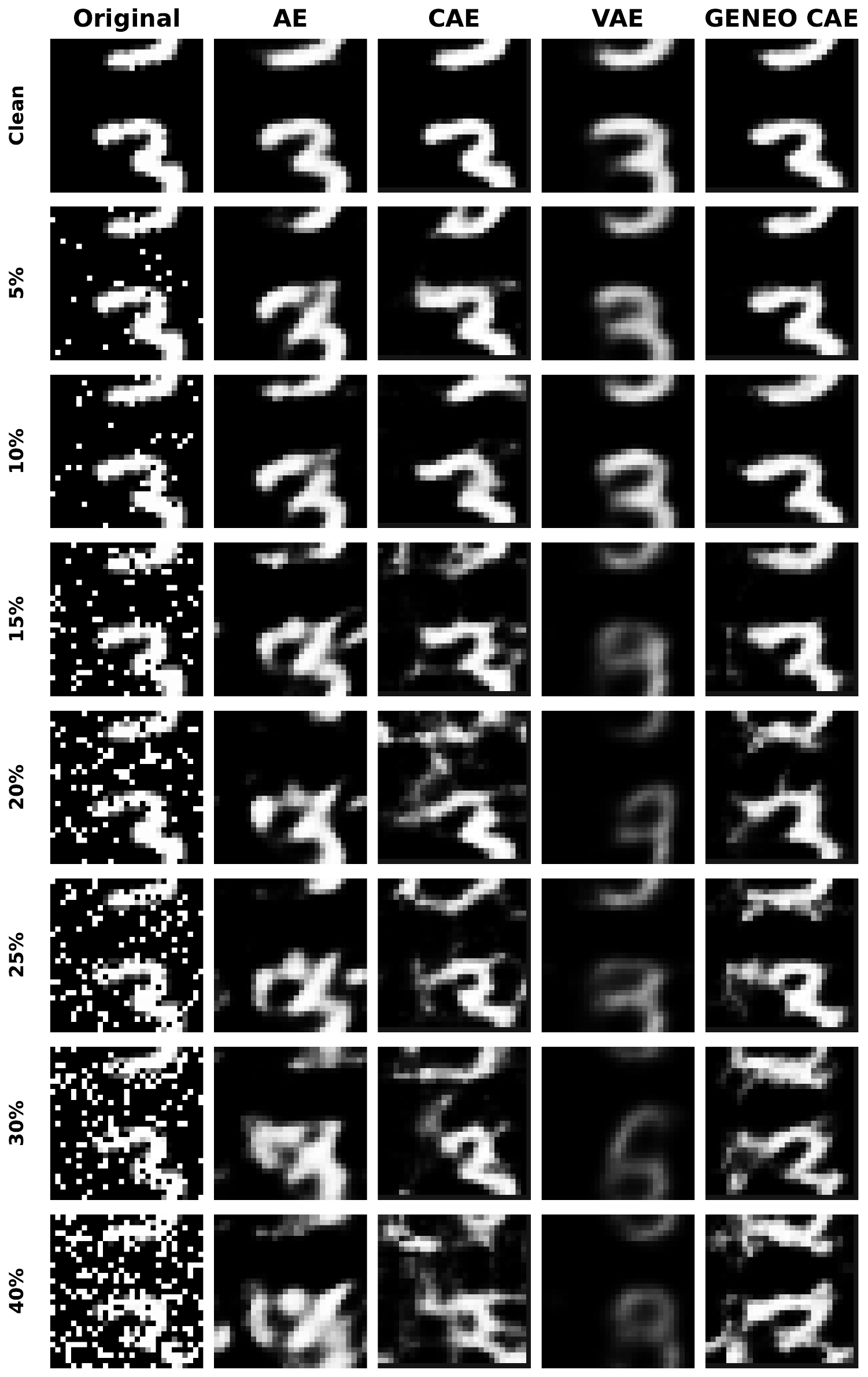}
    \caption{Reconstruction results for test images ID 400 (left) and 500 (right) across increasing salt \& pepper noise levels. Columns show original images and reconstructions from AE, CAE, VAE, and GENEO CAE. The GENEO CAE demonstrates superior ability to preserve digit structure under heavy noise corruption.}
    \label{fig:recon_combined}
\end{figure}

Figure~\ref{fig:recon_combined} provides qualitative visualizations of reconstruction performance on two representative test digits. Even at low noise levels, the GENEO approach produces cleaner images, with fewer blurred or artifact-laden reconstructions. As corruption increases, the advantages of the GENEO layer become more evident. Standard methods quickly lose structural integrity, while GENEO continues to produce recognizable digits, although heavily corrupted. These results validate our hypothesis that GENEO operators extract noise-robust equivariant features that enable superior reconstruction under corruption. The consistent improvement across all evaluation metrics (classification accuracy, reconstruction error, and visual quality) demonstrates that leveraging the mathematical structure of the translation group through GENEO preprocessing provides a principled approach to enhancing autoencoder robustness.

\section{Conclusions}
\label{Conclusions}
In this article we have established a new theorem showing that every linear GENEO between two finite perception pairs can be represented as a sum of the input data, reparameterized by maps weighted through a suitable generalized $T$-permutant measure, under the assumption that the group actions are transitive. This result generalizes a previous theorem that required the input and output data spaces to coincide, and provides a tool for computing every linear GENEO. Such a generalization is not only of theoretical importance but also significantly broadens the applicability of GENEOs in concrete settings. To demonstrate this practical relevance, we concluded the article with an application of the new theorem that enhances the performance of autoencoders.

Future work will address two directions: removing from the theorem the assumption that the domains of the data functions are finite, and developing concrete methods for constructing generalized permutant measures.

\subsection*{Acknowledgements}
F.C., P.F. and N.Q. conducted a portion of their research within the framework of the CNIT WiLab National Laboratory and the WiLab-Huawei Joint Innovation Center. P.F.'s work received partial support from INdAM-GNSAGA, the COST Action CaLISTA, and the HORIZON Research and Innovation Action PANDORA. F.C. was affiliated with the University of Pisa and CNR-ISTI (Pisa) and with Inria, Centre Inria d'Université Côte d'Azur, during the preparation of this work.

\section*{Declarations}

\subsection*{Funding}
P.F.'s work received partial support from INdAM-GNSAGA, the COST Action CaLISTA, and the HORIZON Research and Innovation Action PANDORA. Part of this research was conducted within the framework of the CNIT WiLab National Laboratory and the WiLab-Huawei Joint Innovation Center.

\subsection*{Conflict of Interest}
The authors declare that they have no conflict of interest.

\subsection*{Ethics approval and consent to participate}
Not applicable.

\subsection*{Consent for publication}
Not applicable.

\subsection*{Data availability}
No datasets were generated during the current study. The dataset analysed is available at MNIST \cite{lecun1998mnist}.

\subsection*{Code availability}
The experiments and applications can be reproduced by installing the open-source Python package, available at \url{https://doi.org/10.5281/zenodo.18230049}.

\subsection*{Author contribution}
All authors, Francesco Conti, Patrizio Frosini and Nicola Quercioli, contributed equally to the conception and design of the study, to the development of the theoretical results, and to the writing and revision of the manuscript. Francesco Conti implemented the code and performed the numerical experiments. All authors read and approved the final manuscript.

\end{document}